\documentclass{amsart}

\usepackage{amsmath}

\usepackage{amssymb}
\usepackage{amscd}

\numberwithin{equation}{section}

\long\def\combarak#1{\ifdraft{\sn #1 }\else\ignorespaces\fi}

\newif\ifdraft\drafttrue

\newif\ifdraft\draftfalse

\font\sn = cmssi8 scaled \magstep0

\newcommand{\ebf}{\mathbf{e}}

\newcommand\name[1]{\label{#1}{\ifdraft{\sn [#1]}\else\ignorespaces\fi}}

\newcommand\eq[2]{{\ifdraft{\ \tt [#1]}\else\ignorespaces\fi}\begin{equation}\label{eq:#1}{#2}\end{equation}}

\newcommand {\equ}[1]     {\eqref{eq:#1}}

\newcommand{\EE}{{\mathcal{E}}}
\newcommand{\Hyp}{{\mathbb {H}}}

\newcommand{\R}{{\mathbb{R}}}

\newcommand{\Z}{{\mathbb{Z}}}

\newcommand{\C}{{\mathbb{C}}}

\newcommand{\N}{{\mathbb{N}}}

\newcommand{\PSL}{\operatorname{PSL}}

\newcommand {\ignore}[1]

\newcommand{\sm}{\smallsetminus}

\newcommand{\vre}{\varepsilon}

\usepackage{enumerate}
\usepackage{amssymb}
\usepackage{amsmath}
\usepackage{amscd}
\usepackage{amsthm}
\usepackage{amsfonts}
\usepackage{graphicx}
\usepackage[all]{xy}
\usepackage{verbatim}
\usepackage{hyperref}
\usepackage{epstopdf}
\usepackage[utf8]{inputenc}
\usepackage{bbm}
\usepackage{esint}
\usepackage{tikz}

\newtheorem{theorem}{Theorem}
\newtheorem{proposition}[theorem]{Proposition}
\newtheorem{lemma}[theorem]{Lemma}

\theoremstyle{definition}

\newtheorem{corollary}[theorem]{Corollary}
\theoremstyle{definition}\newtheorem{definition}[theorem]{Definition}

\theoremstyle{definition}

\theoremstyle{definition}
\theoremstyle{definition}\newtheorem{remark}[theorem]{Remark}
\theoremstyle{definition}
\numberwithin{theorem}{section}

\newcommand{\vol}{\mathrm{Vol}}

\newcommand{\eps}{\varepsilon}

\newcommand{\cC}{\mathcal{C}}

\newcommand{\cE}{\mathcal{E}}

\newcommand{\cO}{\mathcal{O}}

\newcommand{\bR}{\mathbb{R}}
\newcommand{\bZ}{\mathbb{Z}}

\newcommand{\bH}{\mathbb{H}}

\newcommand{\SL}{\operatorname{SL}}

\newcommand{\SO}[1][d]{\operatorname{SO}_{#1}}

\newcommand\norm[1]{\left\|#1\right\|} 
\newcommand\set[1]{\left\{#1\right\}} 

\newcommand\on[1]{\operatorname{#1}} 
\newcommand\diag[1]{\operatorname{diag}\left(#1\right)}

\newcommand{\re}{\mathrm{Re}}
\newcommand{\im}{\mathrm{Im}}

\newcommand{\ra}{\rightarrow}

\newcommand{\onto}{\xymatrix{\ar@{>>}[r]&}}
\newcommand{\da}[4]{\xymatrix{#1 \ar@<.5ex>[r]^{#2} \ar@<-.5ex>[r]_{#3} & #4}}

\font\sn = cmssi8 scaled \magstep0

\newcounter{constk}

\newcounter{consta}[section]

\newcounter{constc}[section]

\newcounter{constt}[section]

\title [Effective counting for lattice orbits]{Effective counting for discrete lattice orbits in the plane via
Eisenstein series}

\author{Claire Burrin}

\address{put address here}
\email{put email here}

\author{Amos Nevo}

\address{Department of Mathematics, Technion } 
\email{put email here}

\author{Rene R\"uhr}
\address{Department of Mathematics, Tel Aviv University, Tel Aviv, Israel } 
\email{put email here}

\author{Barak Weiss}
\address{Department of Mathematics, Tel Aviv University, Tel Aviv, Israel} 
\email{barakw@post.tau.ac.il}

\date{\today}

\begin{document}

\begin{abstract}
In 1989 Veech 
showed that for the flat surface
formed by gluing opposite sides of two regular $n$-gons, the set $Y
\subset \R^2$ of saddle
connection holonomy vectors satisfies a quadratic growth estimate $|\{y
\in Y: \|y\|\leq 
R\}| \sim c_YR^2 $, and computed the constant $c_Y$. 
In 
1992 
he recorded an
observation of Sarnak that gives an error estimate 
 $|\{y
\in Y: \|y\|\leq 
R\}| = c_YR^2 + O\left(R^{\frac{4}{3}}\right)$ in the asymptotics. Both
Veech's proof of quadratic 
growth, and Sarnak's error estimate, rely on the theory of Eisenstein
series, and are valid in the wider context of counting points in discrete orbits
for the linear action of a lattice in $\SL_2(\R)$ on the
plane. In this paper we expose this technique and use it to obtain 
the following results. 
For lattices $\Gamma$ with trivial residual spectrum, we recover the
error estimate 
$O\left(R^{\frac43}
\right)$, with a simpler proof. Extending this argument to more
general shapes, and using twisted Eisenstein series, for 
sectors $\mathcal{S}_{\alpha,\beta} = \{r e^{\mathbf{i}\theta} :
r>0, \alpha \leq \theta \leq  \alpha+\beta\}$ we prove an error estimate 
$$|\{y \in Y: y \in \mathcal{S}_{\alpha, \beta}, \|y\| \leq R\}| =
c_Y\frac{\beta}{2\pi} \, R^2 + O_{ \vre}\left (R^{\frac{8}{5}
}
\right).$$ 
For dilations of smooth star bodies $R\cdot B_\psi =\{r 
e^{\mathbf{i} \theta}: 0 \leq r \leq R \psi(\theta)\}$, where $R>0$
and $\psi$ is
smooth, we prove an estimate  
$$|\{y \in Y: y \in R \cdot B_{\psi}\}|=
c_{Y, \psi} R^2 + O_{\psi, \vre}\left(R^{\frac{12}{7}
}\right). $$
\end{abstract}

\maketitle



\begin{center}
   \emph{Dedicated with admiration to the memory of Bill Veech}
\end{center}

\section{Introduction}
We recall the Gauss circle problem, which aims to provide an estimate
for the cardinality $\left|B \cap \Z^2 \right|$ of the intersection
of a large ball $B$ in the 
plane with the integer lattice. The estimate 
$$\left| B(0,R) \cap \Z^2 \right| = \pi
R^2 + O(R)$$ 
is easy
to prove and is attributed to Gauss (here $B(x,r) \subset \R^2$ is the Euclidean
ball of radius $r$ around $x$). There have been several
improvements to the error
term and this is still the topic of intense investigation (see
\cite{Ivic} for a recent survey). A more general
problem in the same vein aims to replace the set $\Z^2$ with another
discrete set $Y$, and replace large balls $B$ with more general sets. 
For more general sets $Y$, the first step is establishing  {\em quadratic
growth}, i.e.\ showing $\left| B(0,R) \cap Y \right| = c_YR^2 + o(R^2)$ for some
$c_Y>0$, and this can already be very challenging. In cases where quadratic
growth has been established, the natural next questions are to evaluate
the {\em quadratic growth constant} $c_Y$, and to obtain
error estimates. A well-studied example is when $Y$ is the
set of primitive points in $\Z^2$, which is a discrete orbit
under the group $\SL_2(\Z)$. This paper is concerned with the case in which $Y$ is a
discrete orbit for a lattice in $G = \SL_2(\R)$ acting on the
plane. An important contribution to the study of these discrete orbits
was made by Veech in a 
celebrated 1989 paper \cite{Veech Eisenstein}, and in the subsequent
papers \cite{Veech regular, Veech Siegel}. We begin by recalling
the context of Veech's work. 

A translation surface is a compact oriented surface equipped with a
translation structure. Since the main results of this paper
will not involve translation surfaces, we omit the precise
definitions, referring the interested reader to the surveys
\cite{Vorobets, MT,  Zorich survey}. For any translation surface $M$, the collection of
{\em holonomy vectors of saddle connections} is a discrete set $Y_M$
in $\R^2$, consisting of planar holonomies of certain straightline
paths on $M$. The group $G$ acts on a moduli space of translation
surfaces, as well as on the plane by linear transformations,
satisfying an equivariance property $Y_{gM} = gY_M$. For any $M$, its
{\em stabilizer group} (or Veech   group) is 
$$
\Gamma_M = \{g \in G: gM=M\}.
$$
If $\Gamma_M$ is a {\em lattice} in $G$, i.e.\ is discrete and of finite
covolume, then $M$ is called a {\em lattice surface} (or Veech
surface). These lattices are non-uniform and thus have discrete orbits
in the plane. 
Here is a summary of the results of \cite{Veech Eisenstein} which are
relevant to this paper. 
\begin{theorem}[Veech 1989]\name{thm: Veech 89}
\begin{itemize}
\item[(a)]
The surfaces $M_n$ obtained by gluing sides in two copies of a regular
$n$-gon are lattice surfaces, and the corresponding lattice is
non-arithmetic unless $n \in \{3,4,6\}$. 
\item[(b)]
For lattice surfaces, $Y_M$ is a finite union of $\Gamma_M$-orbits. 
\item[(c)]
Discrete orbits of lattices in $G$
acting on the plane, satisfy quadratic growth. In particular, the sets
$Y_M$ satisfy quadratic growth when $M$ is a lattice surface. 
\item[(d)] The quadratic growth constants for the surfaces $M_n$ in
  (a) are computed. 
\end{itemize}
\end{theorem}

Veech proved statement (c) by reducing the problem to previous work in
analytic number theory. We will review this below in \S \ref{sec:
  eisenstein}. He also computed quadratic growth constants for the
examples in statement (a), and in 
\cite{Veech regular}, computed quadratic growth constants for
more examples. Veech revisited statement (c) in \cite{Veech 
  Siegel}, where he introduced a number of techniques which make it
possible to establish quadratic growth in more general situations, and
compute quadratic growth constants. Among other things he also
reproved (c) by ergodic methods, in particular using an
ergodic-theoretic tool of Eskin and
McMullen \cite{EM}. Another ergodic-theoretic proof
of (c) was given by Gutkin and 
Judge in \cite{GJ}, also using ideas of \cite{EM}. In subsequent work,
Eskin and Masur \cite{EskinMasur} improved 
on Veech \cite{Veech Siegel} and proved 
that almost every translation $M$ (with respect to the natural measures on
the moduli spaces of translation surfaces) satisfies quadratic
growth. Their arguments are also ergodic-theoretic and rely on an
ergodic theorem appearing in \cite{Nevo17}. 

In the presence of some spectral estimates, it is possible to improve
on quadratic growth by establishing {\em effective quadratic growth},
by which we mean proving an error term of the form 
\eq{eq: effective quadratic}{
\left|B(0, R) \cap Y \right| = c_YR^2 + O\left(R^{2-\delta} \right)
}
for some $\delta>0$. In \cite{NRW}, relying on spectral estimates
established in \cite{AGY, AG}, such an error bound was given for
almost every translation surface. In particular, the results of
\cite{NRW} imply effective quadratic growth for Veech
surfaces. However the constant $\delta$ appearing in \cite{NRW} is far
from optimal, and a much better error estimate for the case of lattice
surfaces has long been known to
experts. In fact, already in \cite[Remark 1.12]{Veech regular}, Veech
included the remark (which he attributed to Sarnak) that work of
Selberg and Good can be used to prove to an
estimate of the form 
\eq{eq: effective error Good}{
\left| B(0, R) \cap Y \right | = c_YR^2 + O\left( R^{\frac{4}{3}} \right), 
}
where $Y$ is the orbit of tempered lattice in $G$ (see \S \ref{sec: statements}).

An initial goal of this paper was to provide an exposition of
the method sketched in \cite[Remark 1.12]{Veech regular}, specifically
for the benefit of those who might be familiar with
ergodic-theoretic counting techniques but not with the techniques used
in analytic number theory. While studying this topic, the authors
obtained several extensions and improvements. Thus the
paper acquired an additional goal of proving these new results;
however we believe that a survey on these matters has not lost its
relevance, and we chose to write our paper on the level of a tutorial.

The structure of the paper is as follows. In \S\ref{sec: statements} we define the
objects which will be the focus of our discussion and state our
results, comparing our new results with those which were obtained by
previous authors (or could be easily deduced from their work). Specifically we define the
class of {\em tempered} lattices and the larger class of {\em lattices with trivial
  residual spectrum}, which are the subgroups for which the relevant
spectral estimates are as strong as one could hope for. As we will
explain, our improvements concern counting points in more general
shapes than Euclidean balls; e.g. sectors or dilates of star-shaped
bodies. In \S \ref{sec: eisenstein} we define Eisenstein series and
collect some results about them. We also explain how Veech
obtained statement (c) of Theorem \ref{thm: Veech 89}. In \S\ref{sec: expository}
we prove the bound \equ{eq: effective error Good} (see
Theorem \ref{thm: good}) for counting in balls, and for lattices with trivial
residual spectrum. Our work bypasses difficult work of Good
\cite{Good} by taking advantage of the fact that in our particular
setting, counting can be achieved by making a contour shift of a
truncated Eisenstein series to the critical line, for a general
lattice. This strategy is classical in analytical number
theory (see e.g. \cite{Davenport}), and indeed goes back to the proof of the Prime Number Theorem,
but in our situation requires an extra averaging argument, see
Proposition \ref{prop: for  
  contour shift}. 
In \S\ref{sec: new} we use this idea to prove our
improvements. Although \S\ref{sec: new} is the one containing the 
proof of the new results, its proof uses the ideas involved in proving
earlier results, and so we do not recommend starting with \S\ref{sec:
  new}. The results of \S 
\ref{sec: expository} and \S \ref{sec: new} both rely on reducing
counting problems to fundamental estimates about Eisenstein
series, which are collected in \S \ref{sec: eisenstein}, and whose
proofs we do not explain. 

\medskip

{\bf Acknowledgements.}
The authors gratefully acknowledge the support of ISF grant 2095/15,
BSF grant 2016256, SNF grant 168823, and ERC starter grant HD-APP 754475.    

CB thanks Ze'ev Rudnick for making possible a visit to Tel Aviv University in
January 2017. The authors thank Avner Kiro, Morten Risager, Ze'ev Rudnick and Andreas
Str\"ombergsson for very helpful discussions. 

\section{Definitions and statement of results}\name{sec: statements}
In this section we set our notation, recall certain preliminary
results, and state our results. 

\subsection{Some actions and subgroups of $G = \SL_2(\R)$}
Recall that $G$ acts on the left on the upper half-plane $\Hyp =
\{z \in \C : \im(z)>0\}$ and on the
plane $\R^2$ respectively, by the
rules
$$
\left(\begin{matrix} a & b \\ c & d  \end{matrix} \right)z =
\frac{az+b}{cz+d} \ \ \text{ and }  \left(\begin{matrix} a & b \\ c & d  \end{matrix} \right) 
\left(\begin{matrix} x\\ y  \end{matrix} \right)=
\left(\begin{matrix} a x+ by \\ cx + d y \end{matrix} \right). 
$$
The $G$-action on $\Hyp$ preserves the hyperbolic metric
$ds^2 = \frac{dx^2+dy^2}{y^2}$ and hence the
hyperbolic area form $
\frac{dx dy}{y^2}$. Let 
$$K = \mathrm{SO}_2(\R) = \left\{r_\theta : \theta \in [0,
  2\pi]\right \}, \text{ where } r_\theta =  \left(\begin{matrix} \cos
    \theta & -\sin \theta \\ 
\sin \theta & \cos \theta \end{matrix} \right).$$
Let $\mathbf{i} = \sqrt{-1}$ so that $K$ is the stabilizer of
$\mathbf{i}$. Also let $\| \cdot \|$ be the Euclidean
norm on $\R^2$; it is also preserved by $K$. Let $\ebf_1 = (1,0)$ so
that the stabilizer of $\ebf_1$ is 
$$
N = \{u_s : s\in \R\},
\text{ where } u_s = \left( \begin{matrix}1 &
  s \\ 0 & 1 \end{matrix} \right).
$$

Let $\Gamma \subset
G$ be a discrete subgroup. It then acts properly discontinuously on
$\Hyp$. We call $\Gamma$ a {\em lattice} if there is a finite $G$-invariant
measure on $G/\Gamma$ or equivalently, a fundamental domain for the
$\Gamma$-action on $\Hyp$ of finite hyperbolic area. If there is a compact
fundamental domain then $\Gamma$ is called {\em cocompact} or {\em
  uniform}. If $\Gamma$ is a lattice we write $X_\Gamma = \Hyp/\Gamma$,
denote the $G$-invariant measure on $X_\Gamma$ induced by the
hyperbolic area form 
by $\mu_\Gamma$ and write $\mathrm{covol}(\Gamma) =
  \mu_\Gamma(X_\Gamma)$.

A subgroup of $\Gamma$ is called {\em maximal unipotent}
if it is conjugate (in $G$) to the group 
$$
N_0 = \left\{ u_n : n \in \Z\right \} 
$$
and is not properly contained in a subgroup conjugate to $N_0$. For a
lattice $\Gamma \subset G$, the quotient 
$\Hyp/\Gamma$ has a finite number of topological ends called {\em
  cusps}. The number of cusps is zero if and only if $\Gamma$ is
cocompact, and in the non-uniform case, there is a bijection between
cusps and conjugacy classes (in $\Gamma$) of maximal unipotent subgroups.
For a lattice $\Gamma$ and $v \in \R^2 \sm \{0\}$, the orbit
$\Gamma v$ is discrete if and only if the stabilizer of $v$ in
$\Gamma$ is a maximal unipotent group, and we refer to the conjugacy
class of the stabilizer
of $v$ as the {\em cusp corresponding to $\Gamma v$}. In particular a cocompact
lattice has no discrete orbits in its action on $\R^2$. Clearly
$\Gamma v$ is discrete if and only if $\Gamma (tv) = t\Gamma v$ is
discrete for all $t \neq 0$, and hence the number of discrete orbits,
considered up to dilation, is the same as the number of cusps. We are
interested in counting points in discrete orbits for the
$\Gamma$-action on $\R^2$.

\medskip

{\em Warning:} In the
literature, one often works with $\PSL_2(\R) = G/\{\pm \mathrm{Id}\}$,
the group of orientation-preserving isometries of $\Hyp$.
Since we are interested in point sets in the plane, which need not be
invariant under the action of $-\mathrm{Id}$, it will be more natural
for us to work with $G$. This discrepancy may
result in minor deviations with other texts; there will be no
discrepancy whenever $\Gamma$ contains $-\mathrm{Id}$, and for
many counting problems we can reduce to this situation as follows:

\begin{proposition}\name{prop: reduction minus id}
Let $\pi: G \to \PSL_2(\R)$ be the natural projection, and suppose
$\Gamma$ is a lattice in 
$G$, which does not contain 
$-\mathrm{Id}$. Let $\Gamma^{(\pm)} = \pi^{-1}(\pi(\Gamma))$, so that
$\Gamma^{(\pm)}$ is a degree 2 central extension of $\Gamma$. Then either
$\Gamma^{(\pm)} v = \Gamma v$ or 
$\Gamma^{(\pm)} v = \Gamma v\sqcup -\Gamma v$ (a disjoint union). 
\end{proposition}

\begin{proof} Since $-\mathrm{Id}$ and $\Gamma$ generate $\Gamma^{(\pm)}$, we
  have $\Gamma^{(\pm)}v = \Gamma v \cup - \Gamma v$. If this is not a
  disjoint union then there is $u \in \Gamma v$ for which $-u \in 
  \Gamma v$, say $u = \gamma_1 v$ and $-u = \gamma_2 v$. Then 
  $$
-v = -\gamma_1^{-1}(u) = \gamma^{-1}_1(-u) = \gamma^{-1}_1 \gamma_2 v 
$$
so that $\Gamma v = -\Gamma v$. 
  \end{proof}

Let $\Delta$ be the
Laplace-Beltrami differential operator on $\Hyp$, which is expressed
in coordinates as 
\eq{eq: Laplacian}{
\Delta f (x+\mathbf{i} y) = -y^2 \, \left( \frac{\partial^2 f}{\partial x^2}
  +\frac{\partial^2 f}{\partial y^2} \right). 
}
It is not hard to check that $\Delta$ is $G$-invariant and hence
descends to a well-defined differential operator on $X_\Gamma$ which
we continue to denote by $\Delta$. 
The eigenvectors for $\Delta$ which belong to $L^2 \left(X_\Gamma,
\mu_{\Gamma}\right )$ are called {\em Maass forms}. The
corresponding eigenvalues satisfy 
$$
 0 = \lambda_0 < \lambda_1 \leq \lambda_2 \leq \cdots, 
$$
and the {\em nontrivial small eigenvalues} are those satisfying $\lambda_i \in
\left(0, \frac14\right].$

\begin{definition}\name{def: tempered}
We say that $\Gamma$ is {\em tempered} if it has no nontrivial small
eigenvalues. 
\end{definition}

Examples of tempered subgroups
are the Veech groups of the surfaces $M_n$ of Theorem \ref{thm: Veech
  89}(1): Veech showed that they are $(2,n,
  \infty)$ Schwarz triangle groups, these triangle groups were shown
  to be tempered by Sarnak \cite[\S 3]{Sarnak thesis}, and all triangle
  groups were shown to be tempered by Zograf
  \cite{Zograf}. 
All but finitely many non-uniform triangle groups arise as Veech
  groups of lattice surfaces, see \cite{BM, Hooper, Wright Schwarz groups}.

Suppose $\Gamma$ is a non-uniform lattice with $k$ cusps, and for $i=1,
\ldots, k$ choose
$\mathfrak{s}_i\in G$ so that $\Gamma'_i = \mathfrak{s}_i^{-1} \Gamma \mathfrak{s}_i $ contains $
N_0$ as a maximal unipotent subgroup, where the groups $\mathfrak{s}_i N_0
\mathfrak{s}_i^{-1}$ are mutually nonconjugate maximal unipotent subgroups
of $\Gamma$. Let 
$$
N'_0 = \left \{ \begin{matrix} \pm N_0 & \ \ \ \Gamma
    \text{ contains -}\mathrm{Id} \\
N_0 & \ \ \ \text{otherwise} \end{matrix} \right. ,
$$ 
and set
\eq{eq: eisenstein 2}{
E_i (z, s) = \sum_{\gamma \in N'_0 \backslash \Gamma'_i} \im(\gamma'
z)^{s}, \ \ \text{where } \gamma' = \gamma \mathfrak{s}_i^{-1},
}
where $z \in \Hyp, \, s \in \C$ and the sum ranges over any
collection of coset representatives.  
Then $E_i$ is the {\em Eisenstein series corresponding to the $i$-th
  cusp of $\Gamma$.} It will play a major role in our discussion and
will be slowly introduced in \S \ref{sec: eisenstein}. As
we will see, for each fixed $i$ and $z$, 
the sum \equ{eq: eisenstein 2}, considered as a map $s \mapsto
E_i(z,s)$, converges for $\re(s) > 1$ and has a meromorphic
continuation to the complex plane. We use this fact
for the following important definition: 
\begin{definition}\name{def: residually trivial}
Let $\Gamma, \, k, \, i$ be as above. The {\em residual spectrum of
  $E_i(z,s)$} is the set of $s \in (1/2,1)$ for which $s \mapsto E_i(z,
s)$ has a  pole at $s$. If there are no such poles we say that {\em
  the residual spectrum of $E_i(z,s)$ is trivial.} 
\end{definition}
We remark that the choice of $z$ in the above definition is
unimportant as all functions $E_i(z, \cdot)$ have poles at the same
values of $s$. 

If $E_i(z, \cdot)$ has a pole at $s \in (1/2,1)$ then
$\Delta$ has an eigenvalue $\lambda
= s(1-s) \in (0,1/4)$. Thus, if 
$\Gamma$ is tempered then all of its cusps have
trivial residual spectrum. 
With regard to the converse, consider
for instance {\em principal congruence groups} (the principal congruence group of level
$n$ is the group of all matrices in $\SL_2(\Z)$ congruent to
$\mathrm{Id} \mod n$). In this case it is known that any cusp for any
congruence group has trivial residual spectrum (see \cite[Thm. 11.3]{Iwaniec}),
but the question of whether all of these group are tempered is a
famous longstanding open question posed by Selberg. 

With this terminology we will prove: 
\begin{theorem}\name{thm: good}
Suppose $\Gamma$ is a nonuniform lattice in $G$, $Y = \Gamma v$ is a
discrete orbit for which the corresponding Eisenstein series has trivial residual
spectrum. Then there is $c_Y>0$ such that 
\eq{eq: error term}{
\left|B (0, R) \cap Y \right| = c_YR^2 + O\left( R^{\frac{4}{3}
}\right).
}
Moreover, the asymptotic \equ{eq: error term} holds when one replaces $B(0,R)$ with the
dilate $R \cdot \mathbf{E}$ of any centered ellipse $\mathbf{E}$ (with
the constant $c_Y$ and the implicit constant in the $O$-notation
depending on $\mathbf{E}$). 
\end{theorem}

In this result one also obtains a precise formula for the quadratic
growth constant
$c_Y$, and an asymptotic expansion for the error in case the
residual spectrum is not trivial, with one term for every pole at $s
\in \left(\frac12,1\right)$. See Theorem \ref{thm: more detailed
  good}. This error estimate in \equ{eq: error term} is not new. In fact,
as we saw in 
\equ{eq: effective error Good}, 
Sarnak and
Veech could prove it 
already in 1993. However the
proof we will give below will be simpler than the
proof outlined in \cite{Veech regular}, which relies on difficult work
of Good concerning counting results in both $\R^2$ and in $\Hyp$. See
\S \ref{subsec: good} for a more detailed comparison with Good's work. 

An interesting open question is whether the error term in Theorem
\ref{thm: good} is optimal. In this regard, note that for $\Gamma =
\SL_2(\Z)$, the error term in \equ{eq: 
  effective error Good} can be improved to $o(R)$ (see \cite{HN}); the
same is 
true for congruence groups. 
However we are not aware
of any non-arithmetic non-uniform lattices for which a bound better than that
of \equ{eq: effective error Good} is known. 

\subsection{Counting in more general domains}\name{subsec: new}
We turn to new results. In these results we strive to take sets more
general than Euclidean balls in the counting problem, while still
obtaining a good bound for the error.  We will need a further
definition.

For each $\gamma \in G$, we set $c_\gamma = c, d_\gamma =d$ where 
$\gamma = \left(\begin{matrix} a & b \\ c & d\end{matrix}
\right).  
$ Note that $c_\gamma, d_\gamma$ only depend on the coset $N\gamma$.
For each $n \in \Z$, and for $\Gamma, \Gamma_i, \mathfrak{s}_i, N'_0$ as in
the discussion preceding \equ{eq: 
  eisenstein 2}, define a {\em twisted Eisenstein series} 
\eq{eq: twisted eisenstein}{
E_i(z,s)_{2n} = \sum_{\gamma \in N'_0 \backslash \Gamma'_i}
\im(\gamma'z)^s  \left(\frac{c_{\gamma'} z
    + d_{\gamma'}}{| c_{\gamma' } z +
    d_{\gamma' }|} \right)^{2n} \ \  \text{ where }
\gamma' = \gamma \mathfrak{s}_i^{-1}
}
(note that this definition makes sense for any $m$ in place of $2n$
but we are only interested in the even values). 
Once again it is true that $s \mapsto E_i(z, s)_{2n}$ has a meromorphic
continuation to the entire complex plane, whose poles do not depend on
$z$, and we generalize 
Definition \ref{def: residually trivial} as follows:  
\begin{definition}\name{def: residual spectrum twisted}
For $i, n$ as above, the {\em residual spectrum} are 
those $s \in (1/2,1)$ for which $E_i(z,\cdot )_{2n}$ has a pole at $s$. If
there are no such $s$ we say {\em the residual spectrum of $E_i(z,s)_{2n}$ is trivial.}
\end{definition}
Once again it is true that $E_i(z,s)_{2n}$ has finitely many poles and a tempered group
$\Gamma$ has trivial residual spectrum for 
each $i$ and $n$. 

Let $S \subset \R^2$ be a bounded closed set. We say that $S$ is {\em
  star shaped at $0$} if it can be written as 
$$S = \{r(\cos \theta, \sin \theta): \theta \in [0, 2\pi], r \in [0,
\rho(\theta)]\}$$
for $\theta \mapsto \rho(\theta)$ a non-negative bounded $2\pi$-periodic function
of compact support. 

We say that $S$ is a {\em sector} if it is of the above form with the
function $\rho$ the
indicator of a nondegenerate subinterval. We say $S$ is a {\em smooth star shape} if
it is of the above form and $\rho$ is smooth and everywhere positive. We write $R \cdot S$ for the
dilated set $\{Rx: x \in S\}$. 

With these notations we have: 
\begin{theorem}\name{thm: counting in shapes}
Suppose $\Gamma$ is a non-uniform lattice in $G$ containing
$-\mathrm{Id}$, $Y = \Gamma v$ is a 
discrete orbit corresponding to the $i$-th cusp, and suppose that for
each $n$, $E_i(z,s)_{2n}$ has trivial residual
spectrum. Then: 
\begin{itemize}

\item
If $S$ is a smooth star shape then there is $c_{Y,S}>0$ such
that for every $\vre>0$, 
$$
\left| Y \cap R \cdot S \right| = c_{Y,S}R^2 +
O\left(R^{\frac{12}{7}+\vre
}\right).
$$
\item
If $S$ is a sector then there is $c_{Y,S}>0$ such
that 
\eq{eq: counting in sector}{
\left| Y \cap R \cdot S \right| = c_{Y,S}R^2 +
O\left(R^{\frac{8}{5}
}\right);
}
moreover, the asymptotic \equ{eq: counting in sector} is also valid
if one replaces $S$ with a sector in a centered ellipse (i.e.\ the
image of a sector under an invertible linear map), with implicit
constants depending also on the ellipse. 
\end{itemize}
\end{theorem}
In the above results, the quadratic growth constants can be written
down explicitly and the implicit constants depend on the sets $S$.

\subsection{Relation to the work of Good}\name{subsec: good}
Let $G = KAN$ be the Iwasawa decomposition of $G$, that is $K, A, N$
are respectively the subgroup of orthogonal, diagonal, and unipotent
upper triangular matrices, and let $\Gamma$ be a non-uniform lattice
normalized so that it contains $N_0$ as a maximal unipotent subgroup. 
The counting problem considered in Theorem
\ref{thm: good} can be thought of as a
             counting problem in the double coset space
$$
\mathfrak{S}=( \Gamma \cap K)\backslash \,  \Gamma \, \slash (\Gamma\cap N).
$$
In fact, one can easily verify that 
\begin{align*}
\left| \Gamma \ebf_1 \cap B(0, R)  \right |&= \left|\{[\gamma]\in
                                             \mathfrak{S}:\|\gamma \ebf_1\|\leq
                                             R\} \right |\\
&= \left|\left \{ [\gamma]\in \mathfrak{S}: \im(\gamma^{-1}i)\geq R^{-2}\right\} \right |\\
&= \left| \left \{ (\theta\mod 2\pi, y, x\mod 1): [r_\theta a_y u_x]\in
  \mathfrak{S},\ y\geq R^{-2}\right\} \right |,
\end{align*}
where $a_y = \mathrm{diag}(y, y^{-1})$ and the last identification
relies on the Iwasawa decomposition of 
$G$. Fix $y>0$, and set 
$$
\mathcal{K}_y = \{ (\theta\mod 2\pi,x\mod 1): [r_\theta a_y n_x ]\in \mathfrak{S}\}.
$$
The character sums, for $m, n\in\Z, y>0$, 
$$
\mathbf{S}(m,n,y) = \sum_{(\theta,x)\in \mathcal{K}_y} e^{\mathbf{i}m\theta}e^{2\pi \mathbf{i}nx}
$$
indexed over the above Iwasawa double coset decomposition are a
natural generalization of the classical Kloosterman sums from number
theory (which appeared already in work of Poincar\'e about Fourier expansions of Eisenstein
series, see \cite{Poincare}). Note in 
particular that 
$$
|\Gamma \ebf_1\cap B(0,R)| = \sum_{y\geq R^{-2}} \mathbf{S}(0,0,y).
$$
In \cite[Thm. 4]{Good}, Good proved bounds on the asymptotic growth of sums of
various 
generalizations of Kloosterman sums as above, meaning over various
double coset decompositions of $\Gamma$ in $G$. This corresponds to
the problem we have discussed above (counting for $\Gamma$-orbits
in the plane) as well as other counting problems such as
$\Gamma$-orbits in $\Hyp$ and in the space of geodesics. For the case
of the linear action on the plane, which is the one of interest here,
Good obtains 
$$
\sum_{y\geq R^{-2}}\sum_{(\theta,x)\in K_y} e^{\mathbf{i}m\theta}e^{2\pi \mathbf{i}nx} =
c_{m,n} R^2 +
O_{m,n}\left(R^{\frac43} \right),
$$
where $c_{m,n} > 0$ if and only if $m=n=0$. The
dependence of implicit constants in the remainder term on $m, n$ is
however not worked out explicitly 
(and difficult to trace over the 100 pages of build-up Good relies on
to prove this asymptotic). If it were, one would be able to deduce results
similar to Theorem \ref{thm: counting in shapes} from \cite{Good}.

\subsection{Counting in still more general (well-rounded)
  domains}\name{subsec: well rounded domain counting}
 A common assumption in the theory of counting lattice points in a
 family of domains in a Lie group is that of well-roundedness,
 introduced in  \cite{DRS} and \cite{EM}. In \cite{GN12} this
 assumption was combined with an estimate of the spectral gap that
 arises in the automorphic representation of  
$G$ on $L^2_0(G/\Gamma)$ to prove an effective estimate for the
lattice point count. We will show that the problem we consider, namely
counting points in discrete orbits for the linear action of a lattice on the plane, can
be reduced to the lattice point counting problem for domains in
$\SL_2(\R)$. This will allow us to count orbit points in more general
sets in the plane using just the existence of a spectral gap, but this
additional generality compromises the error estimate, leading to
bounds which are 
inferior to the ones stated above. Thus the techniques we describe in
this subsection are applicable in more general situations, but lead
to weaker bounds. 

\ignore{
We recall the following

\begin{definition}(\cite{GN12})
\label{def: well--roundedness}Let $G$ be a connected Lie group with Haar
measure $m_{G}$. Assume $\left\{ \cC_{t}\right\} \subset G$ is a
family of bounded Borel sets of positive measure such that $m_{G}\left(\cC_{t}\right)\ra\infty$
as $t\ra\infty$. Let $\cO_{\eta}\subset G$ be the image of a ball
of radius $\eta$ (with respect to the Cartan-Killing norm) in the Lie
algebra under the exponential map. Denote 
\[
\cC_{t}^{+}\left(\eta\right)=\cO_{\eta}\cC_{t}\cO_{\eta}=
\bigcup_{u,v\in\cO_{\eta}}u\,\cC_{t}\,v,\,\,\,\,\,\,   
\cC_{t}^{-}\left(\eta\right)=\bigcap_{u,v\in\cO_{\eta}}u\,\cC_{t}\,v.
\]
The family $\left\{ \cC_{t}\right\} $
is \emph{Lipschitz well-rounded} if there exist positive $c, \eta_{0}, t_{0}$
such that for every $0<\eta\leq\eta_{0}$ and $t\geq t_{0}$,
\[
m_{G}\left(\cC_{t}^{+}\left(\eta\right)\right)\leq
\left(1+c\eta\right)\,m_{G}\left(\cC_{t}^{-}\left(\eta\right)\right).  
\]
\end{definition}
}
We will show:

\begin{theorem}\label{thm:LWR}\name{well-rounded} 
Let  $\Gamma$ be any non-uniform lattice in $\SL_2(\R)$, and $Y = \Gamma v$
any discrete orbit of $\Gamma$ in $\R^2\setminus \set{0}$. Let
$S\subset \R^2$ be a star-shaped domain at $0$ with $\rho(\theta)$ a
non-negative 
piecewise Lipschitz function, 
and let $R \cdot S$ be the dilation of $S$ by a factor of $R$. 
Then, for all $\vre > 0$ 
$$
\left| Y \cap R \cdot S  \right| = C_{Y,S}R^2 +
O\left(R^{q_\Gamma+\vre}\right)\,, 
$$
where the implicit constants in the $O$-notation depend on $\Gamma, Y, S$
and $\vre$, and 
with $q_\Gamma$ depending only on the spectral gap of the automorphic
representation of $G$ on $L^2_0(G/\Gamma)$. In particular, if the
lattice $\Gamma$ is tempered, then  
we can set $q_\Gamma= \frac{7}{4}$.
\end{theorem}

 Note that Theorem \ref{thm:LWR} applies, in particular, to all convex sets with piecewise Lipchitz boundary (containing the origin in their interior), and in particular, to all convex polygons.

\section{A bit of Eisenstein series}\name{sec: eisenstein}
In this section we go into  more details about our main actor, the Eisenstein series
introduced in \equ{eq: eisenstein 2}. We refer the
reader to \cite{Kubota, Hejhal, 
  Terras, Sarnak horocycles} for more information. 
\subsection{Some sums and their relation to the counting
  function}\name{subsec: some sums} For a non-uniform lattice $\Gamma$, a
discrete orbit $Y= \Gamma v$ in the plane, $g \in G$ and $s \in \C$, we set 
\eq{eq: eisenstein}{
E(g, s) = E^{(\Gamma, v)}(g,s)= \sum_{u \in g\Gamma v} \|u\|^{-2s}.
}
Note that in \equ{eq: eisenstein 2}, \equ{eq: eisenstein} we introduce the notation $E$ and $E^{(\Gamma, v)}$ to
denote two different functions, one of which has an argument $z \in \C$
and the other, $g \in G$. This ambiguity is common in the literature
and is explained below, see (\ref{eq:eisensteinagrees}). 
Our first task will be to motivate this new definition, in the context of
the counting problem for points in $\Gamma v$. For the moment we
consider \equ{eq: eisenstein} as a formal sum, postponing the discussion of 
convergence issues.

\medskip

{\em Warning (continued).}
In the literature, there are two conflicting conventions regarding the definition
of $E(g,s)$. What is denoted $E(g,s)$ in \cite{Veech Eisenstein} is
denoted 
$E(g^{-1},s)$ in \cite{Kubota}. We will follow Veech's convention, and
we say more about the 
source of this discrepancy below.

\medskip

Let $\mathbf{N}(g, R) = |B(0, R) \cap gY|$. Considering the measure
$\nu^{(g)} = \sum_{u \in g\Gamma v} \delta_u$, which is a Radon
measure on $\R^2$  (since
$\Gamma v$ is discrete), and considering
the radial function $f^{(s)}(w) = \|w\|^{-2s}$, we have 
$$
E(g,s) = \int_{\R^2} f^{(s)} d\nu^{(g)} = \int_0^\infty \frac{d\mathbf{N}(g,R)}{R^{2s}} ,
$$
where in the last equality we have written a Lebesgue-Stieltjes
integral. Using integration by parts (and recalling that convergence issues
will be addressed further below), we have 
\eq{eq: integration by parts}{
E(g,s) = 2s \int_0^\infty \frac{\mathbf{N}(g,
  R)}{R^2} R^{1-2s} dR.
}
We now recall the definition of the Mellin transform and Mellin
inversion, which are
multiplicative analogues of the Fourier transform and Fourier
inversion. Recall that a Schwartz function is a function $\R \to \R$
which is infinitely differentiable and for which all derivatives decay
to zero at infinity faster
than any power. We will say that $\psi: \R_+ \to \R$ is a Schwartz
function on $\R_+ = (0, \infty)$
if  $f(x) = \psi(e^x)$ is a Schwartz function. The Mellin transform of
a Schwartz function $\psi: \R_+ \to 
\R$ is given by 
\eq{eq: Mellin transform}{
\mathcal{M}\psi : \C \to \C, \ \ \mathcal{M}\psi (s) =
\int_0^\infty \psi(y) y^{s-1} dy,
}
and Mellin inversion says that for $\sigma \in \R$ we have 
\eq{eq: Mellin inversion}{
\psi(y) = \frac{1}{2\pi \mathbf{i}} \int_{\re(s) = \sigma}
\mathcal{M}\psi(s) \, y^{-s} ds.
}
The above formulae follow immediately from the Fourier transform and
Fourier inversion formula, from which they are obtained by a change of
variables $y = e^x$. As we will explain below, under suitable
conditions the formulae extend to
functions which are not Schwartz functions. For the moment we 
proceed considering them as formal identities. 

Comparing equations \equ{eq: integration by parts} and \equ{eq: Mellin
  transform}, and making a change of variables $y = R^{-1}$, we see
that $\frac{E(g,s)}{2s}$ is the Mellin transform 
of the function $y \mapsto \mathbf{N}(g,y^{-1})$ evaluated at $2s$. Applying
Mellin inversion we recover the counting function $\mathbf{N}$ as 
\eq{eq: eisenstein mellin}{
\mathbf{N}(g, R) = \frac{1}{2\pi \mathbf{i}} \int_{\re (s) = \sigma} \frac{E(g,s)}{s}
R^{2s} ds 
}
(this formal
manipulation is given a precise meaning and justified in Corollary \ref{cor:
  convergence} below). The upshot of this discussion is that, at 
least formally, the counting function which we are interested in has
an integral representation in terms of the function
$E(g,s)$. Furthermore, if we know (as will turn out to be the case)
that $s \mapsto E(g,s)$ is holomorphic or meromorphic, then the
integral of \equ{eq: eisenstein mellin} can be evaluated using
standard tools of complex analysis like contour shifts, residue
computations, etc. Using this, after justifying our manipulations we
will indeed be able to obtain a detailed understanding of
$R \mapsto \mathbf{N}(g,R)$ from an understanding of $E(g, s)$. Note also that up
to this point no use has been made of the dependence of all
quantities on the variable $g$. This dependence will not play much of
a role in our discussion, but it is crucial
when one wants to say something about $E$. 

For the benefit of readers not satisfied with this non-rigorous derivation
of \equ{eq: eisenstein mellin}, we include another non-rigorous
derivation. Let $u \in g\Gamma v$ and consider its contribution to
both sides of \equ{eq: eisenstein mellin}. Assume for simplicity that
$g\Gamma v$ does not contain vectors of length precisely $R$, and set
$y = R/\|u\|$, so that $u$ contributes $1$ to $\mathbf{N}(g, R)$ when
$y>1$ and contributes 0 when $y<1$. Recalling \equ{eq: eisenstein},
and exchanging the order of summation
and integration in the right hand side of \equ{eq:
  eisenstein mellin}, we see that each $u$ contributes $\frac{1}{2\pi
  \mathbf{i}} \int_{\re(s)=\sigma} \frac{y^{2s}}{s} ds.$ This integral
is the
limit as $T \to \infty$ of line integrals along the vertical lines
$L_{\sigma,T} = \{\sigma + \mathbf{i}t: -T \leq t \leq T\}$. For each
fixed $T$ we can evaluate
this line integral by Cauchy's integral formula, replacing $L_{\sigma,T}$
with $L_{\zeta,T}$ (the total contribution along the horizontal lines $\im(s) =
\pm T$ becomes negligible as $T \to \infty$), where in case $y>1$, we
let $\zeta \to -\infty$, and  
get a contribution of $1$ due to the pole at the origin, and in case $y<1$
we let $\zeta \to +\infty$ and get a contribution of 0.

\subsection{Simple properties and the relation to Eisenstein series}
Having motivated our interest in the function defined by \equ{eq:
  eisenstein}, we now make the link with the 
functions defined by \equ{eq: eisenstein 2}. Let $\Gamma_v =
\{\gamma \in \Gamma : \gamma v = v\}$. Note that if $g = r_\theta  \diag{y,
y^{-1}} u_s$, where $y>0$ and $r_\theta \in K,\, u_s \in N$ (Iwasawa
decomposition), then $y$ can be detected in both the linear action as
$\|g \ebf_1\|$ and in the action on the upper half plane as
$\im(g^{-1}\mathbf{i})^{-\frac{1}{2}}$. Using this observation, the
following properties follow readily 
from the definition \equ{eq: 
  eisenstein} and from the fact that the Euclidean norm is $K$-invariant. 
\begin{proposition}
\name{prop: simple properties}
\begin{enumerate}
\item
For $r_\theta \in K$ and $\gamma \in \Gamma_v$ we
have $E(r_\theta g\gamma,s) = E(g,s)$.  
\item
For $g, \mathfrak{s} \in G$, if  $\Gamma^{\mathfrak{s}} = \mathfrak{s}^{-1} \Gamma
\mathfrak{s}$ then $E^{(\Gamma^{\mathfrak{s}}, \mathfrak{s}^{-1} v)}(g \mathfrak{s},s) = E^{(\Gamma,
  v)}(g,s)$. 
\item
Suppose $\Gamma$ contains $N_0$ as a maximal unipotent subgroup. 
If $\Gamma$ does not contain -$\mathrm{Id}$ then $E^{(\Gamma, \ebf_1)}(g, s) =
\sum_{\gamma \in N'_0
  \backslash \Gamma} \im(\gamma g^{-1} 
\mathbf{i})^s$; if $\Gamma$ contains -$\mathrm{Id}$ then
$E^{(\Gamma, \ebf_1)}(g, s) = 2 \sum_{\gamma \in N'_0 
  \backslash \Gamma} \im(\gamma g^{-1} 
\mathbf{i})^s$. 
\item
If $v_1, v_2 \in \R^2 \sm \{0\}$ satisfy $v_2 = tv_1$ for $t>0$, then
$E^{(\Gamma, v_1)}(g,s) = t^{-2s} 
E^{(\Gamma, v_2)}(g,s)$. 
\end{enumerate}
\end{proposition}
By property
(1), the dependence of \equ{eq: eisenstein} on $g$ is actually only a dependence on the coset
$Kg$, and we can identify these cosets with $\Hyp$ via $Kg
\leftrightarrow z= g^{-1} \mathbf{i}$ to replace $g$ with $z$. And
with the normalization that the stabilizer of $v$ is $N_0$, we see by
rescaling and using 
properties (2) and (3) that
\begin{equation}\label{eq:eisensteinagrees}
	z = g^{-1} \mathbf{i}  \ \ \implies \ \ E_i(g,s) =
\left\{\begin{matrix} E_i(z,s) & \Gamma \text{ does not contain
      -}\mathrm{Id} \\ 2 E_i(z,s) &
    \text { otherwise}  \end{matrix} \right. . 
\end{equation}
Thus for each non-uniform lattice $\Gamma$ with $k$ cusps, up to the trivial
transformations recorded above, there are $k$ essentially different
functions of this form. They are normalized by conjugating so that
$\Gamma_v = N_0$ and rescaling so that $v = \ebf_1 $. It will become
clearer later why this normalization is convenient. It will also
develop that in order to understand these functions in detail,
it is best not to focus on one of them, but to consider their
properties as a vector valued function $(z,s) \mapsto (E_1(z,s),
\ldots, E_k(z,s)).$ 

The discrepancy between the notation used in \cite{Veech Eisenstein}
and that used in \cite{Kubota} is related to the 
substitution $z = g^{-1} \mathbf{i}$ above. If one followed the
convention of Kubota one would make the substitution $z = g
\mathbf{i}$ instead. The convention of Veech, which we follow, 
gives simpler formulae involving discrete orbits in $\R^2$ and is consistent
with working with the space of left cosets $G/\Gamma$. The
convention of Kubota gives simpler formulae when discussing the action
of $G$ on $\bH$ by M\"obius transformations, and is consistent with
working with right cosets $\Gamma \backslash G$. Thus the discrepancy
between these notations is collateral damage in a 
larger battle. 

We now explain our interest in the twisted Eisenstein
series \equ{eq: twisted eisenstein}. Above we motivated Eisenstein
series by explaining its 
relation to the counting problem in the plane, where each orbit point
is assigned the same mass 1. In this application the counting function is
$K$-invariant, and so we can equivalently view the first parameter of
the Eisenstein series as ranging in 
 $g \in G$ or in $z \in \Hyp$ (as 
in the preceding paragraph). In more general situations it is desirable to
assign different masses to different points, and in particular allow
functions which depend on $g$ rather than on the coset $Kg$. This will
arise when we deal with more refined counting problems as in Theorem
\ref{thm: counting in shapes}, and also arises in many other problems
of geometric origin. 

For a vector $u \in \R^2 \sm \{0\} \cong \C \sm
\{0\}$ we define polar coordinates 
$u = \|u \| e^{\mathbf{i}
  \theta_u},$ where $\theta_u \in 
\R/2\pi\Z$. Let 
$\Gamma v$ be a discrete orbit corresponding to the $i$-th cusp $\Gamma$,
normalized so that $v = \mathfrak{s}_i \ebf_1$, and set
\eq{eq: twisted eisenstein 2}{
E_i(g,s)_n = \sum_{u \in g \Gamma v} \|u\|^{-2s} e^{ -\mathbf{i} n\theta_{u}}.
}
Note that these functions vanish for $n$ odd when $\Gamma$ contains
$-\mathrm{Id}$. 
 It is not hard to formulate an analogue of
Proposition \ref{prop: simple properties} and, by comparing \equ{eq:
  twisted eisenstein} and
\equ{eq: twisted eisenstein 2}, to verify that
\eq{eq: choice of g}{
  E_i(z,s)_{2n} = E_i(g,s)_{2n} \ \ \text{ when } z = x + \mathbf{i}y
  \text{ and } g = \left(\begin{matrix}
  y^{-1/2} & -xy^{-1/2} \\ 0 & y^{1/2}\end{matrix} \right).}
Note that the choice of $g$ in \equ{eq: choice of g} ensures $z = g^{-1} 
\mathbf{i}$, and if we choose another $g$ with this property, this
will only affect $E_i(g,s)_{2n}$ by multiplication with a complex
number of modulus 1. 

\medskip

{\em Warning (continued):}
In \equ{eq: twisted eisenstein 2}, it would have been more
natural, and consistent with the Veech convention mentioned after
Proposition \ref{prop: simple properties}, to
define the Eisenstein series using $e^{\mathbf{i}n 
  \theta_u}$ instead of $e^{-\mathbf{i}n
  \theta_u}$. However this would have made it necessary to introduce a
change of signs in \equ{eq: twisted eisenstein} and would have caused
a discrepancy between our notation and that of \cite{Selberg
  gottingen, Sarnak horocycles}.

\medskip

Treating more general weights of points on the plane also leads to the
$\Theta$-transform which we will discuss in \S \ref{subsec: theta transform}.  

\subsection{Convergence properties}
We now begin our discussion of convergence properties of the various
series introduced so far, and give a more rigorous justification of
\equ{eq: eisenstein mellin}. Convergence rests on the following
weak (and standard) counting estimate. 

\begin{proposition}\name{prop: quadratic upper bound}
For each $g \in G$ we have $\mathbf{N}(g,
R) = O(R^2)$. Moreover the implicit constant can be taken to be
independent of $g$. 
\end{proposition}

\begin{proof} We will give a simple proof in which the implicit
  constant will appear to depend on $g$. For a similar but more careful proof,
  which explains how to take the constant independent of $g$, see \cite[Lemma
  16.10]{Veech Siegel}.  

Make
a change of variables so that $v = \ebf_1$ and $\Gamma_v = N_0$, and
compare the actions on $\R^2$ and $\Hyp$. Let $\gamma v \in \Gamma v$,
and using Iwasawa decomposition, write
$$\gamma = r_\theta a_\gamma u_s.$$
We can choose $\gamma \mod N_0$ so that $u_s$ is bounded, and the
condition $\|\gamma v \| \leq R$ is equivalent to $y_\gamma \leq R$,
where $a_\gamma = \diag{y_\gamma, y^{-1}_\gamma}$.
Now apply $\gamma^{-1}$ to
$\mathbf{i}$. Since $r_\theta$ preserves $\mathbf{i}$, 
$a_\gamma^{-1} \mathbf{i} = \frac{1}{y_\gamma^{2}} \mathbf{i}$, and $u_s$ is
bounded, we see that
$\gamma^{-1}\mathbf{i}$ is contained in a set $A_R$ which is an $r$-neighborhood of
the ray $\left\{t\mathbf{i} : t \geq R^{-2} \right\}$, for some $r>0$ independent
of $R$. The hyperbolic area of $A_R$ is $O(R^2)$. 
On the other hand, since $\Gamma \mathbf{i} \subset \Hyp$ is discrete and since $G$
acts on $\Hyp$ by isometries, there is 
$r>0$ small enough so that the balls of radius $r$ around
points of $\Gamma \mathbf{i}$ are disjoint. So the intersection of $A_R$
with $\Gamma \mathbf{i}$ contains $O(R^2)$ points.
\end{proof}

\begin{corollary}\name{cor: convergence}
\combarak{I don't know why latex typesets this in roman rather than
  italicized. When I change it to a proposition, it works well, but I
  don't see any difference in the declarations.}
 The quantities in \equ{eq:
  eisenstein 2}, \equ{eq:
  eisenstein} and \equ{eq: integration by parts} converge absolutely on
$\{ \re(s)>1\}$ and converge 
uniformly on compact subsets of $\{ \re(s) > 1\}$. For any $\zeta >1$,
\equ{eq: eisenstein 2} and 
\equ{eq: eisenstein} are bounded on $\{ \re(s) \geq \zeta\}$ by a bound
which can be taken to be uniform as $z$ and $g$ vary in a compact
set. 
For fixed $g$, and for $\sigma>1$, 
\eq{eq: eisenstein mellin 2}{
\frac{\mathbf{N}(g, R^+)+\mathbf{N}(g, R^-)}{2} = \frac{1}{2\pi} \lim_{T \to \infty}
\int_{- T}^T \frac{E(g, \sigma + \mathbf{i} 
  t)}{\sigma + \mathbf{i} t} R^{2(\sigma + \mathbf{i}t)} \, dt,
}
where $\mathbf{N}(g, R^+), \mathbf{N}(g, R^-)$ denote the one-sided limits of $\mathbf{N}(g,R)$ as $x \to
R$. 
\end{corollary}
\begin{proof}
The claim regarding \equ{eq:
  eisenstein 2}, \equ{eq:
  eisenstein} and \equ{eq: integration by parts} follows easily from
Proposition \ref{prop: quadratic upper bound}. For instance, for 
\equ{eq: eisenstein}, split the sum into sums over the `rings' 
$$\left\{ w \in \Gamma v :
\| w \| \in \left[2^n, 2^{n+1} \right] \right\}$$ 
for $n \in \N$. 

For \equ{eq: eisenstein mellin 2}, fix $g \in G$ and $\sigma_0>2$, and
define the function $\psi_{\sigma_0}(\tau) = \mathbf{N}(g, e^{\tau})e^{-\sigma_0 \tau}.$
Then $\psi_{\sigma_0}(\tau)$ has finitely many discontinuities on every
bounded interval, with well-defined one-sided limits, and vanishes
when we take 
$\tau \to -\infty$ (by discreteness of $Y$). Also, by Proposition \ref{prop: quadratic upper
  bound}, we have $\psi_{\sigma_0}(\tau) = O(e^{(2-\sigma_0)\tau})$ as $\tau \to
\infty$, and hence $\psi_{\sigma_0} \in L^2(\R) \cap L^1(\R)$. Write 
$$
\widehat{\psi_{\sigma_0}}(u) = \int_{-\infty}^\infty
\psi_{\sigma_0}(\tau)e^{-2\pi \mathbf{i} u\tau} d\tau 
$$
for the Fourier transform of $\psi_{\sigma_0}$. Using \equ{eq:
  integration by parts} and making changes of variables $R= e^\tau,\, 2s
= \sigma_0+ 2\pi \mathbf{i} u$ we have 
$$
\widehat{\psi_{\sigma_0}}(u) = \frac{E(g,s)}{2s}, \ \text{where }
\re(s) > 1.
$$
Then by
Fourier inversion (see e.g.\ \cite[Ex. 1.2.7]{Terras}), for all $\tau
\in \R$ we have  
$$
\frac{\psi^+_{\sigma_0}(\tau)+ \psi^-_{\sigma_0}(\tau)}{2}= \lim_{T\to \infty}
\int_{-T}^T \widehat{\psi_{\sigma_0}}(u) e^{2\pi \mathbf{i} u\tau}du,  
$$
and hence (with the changes of variables $\sigma = \sigma_0/2, \, t =
\pi u, \, R=e^\tau, \, s = \sigma+ \mathbf{i} t$)
\[
\begin{split}
& \frac{\mathbf{N}(g, R^+)+\mathbf{N}(g, R^-)}{2}   = R^{\sigma_0}
\frac{\psi^+_{\sigma_0}(\log R)+
  \psi^-_{\sigma_0}(\log R)}{2} \\ = & \lim_{T\to \infty}
\int_{-T}^T R^{\sigma_0} \frac{E(g,\sigma  + \pi \mathbf{i}
  u)}{\sigma_0 + 2 \pi \mathbf{i}u}  R^{2\pi \mathbf{i} u}du  
= \frac{1}{2\pi} \lim_{T\to \infty} \int_{-T}^T \frac{E(g, \sigma+
  \mathbf{i}t)}{\sigma+ \mathbf{i}t}
R^{2(\sigma+\mathbf{i}t)} dt.
\end{split}
\]
\end{proof}

\subsection{Selberg's results: meromorphic continuation and functional
equation} \name{subsec: Selberg results}
We now move beyond elementary results and come to much deeper results
about Eisenstein series. Most of 
these results are due to celebrated 
work of Selberg, see \cite{Selberg Indian, Selberg gottingen,
  Hejhal, Kubota} (the
introduction to \cite{Selberg gottingen} contains some historical
notes). The proofs exploit the dependence of $E(z,s)$ on the variable
$z$, and we content ourselves with two comments, in order to clarify the
connection with objects appearing in the preceding sections.

For $s \in \C$, the functions $f(x + \mathbf{i}y) = y^s$ clearly
satisfy $\Delta f = s(1-s) f$, i.e.\ are eigenfunctions for the Laplace-Beltrami
operator. Since $\Delta$ is $G$-invariant, formula \equ{eq:
  eisenstein} shows that for fixed $s$, the Eisenstein series also
gives rise (at least formally) to a Laplace-Beltrami eigenvector $z
\mapsto E_i(z,s)$, thus furnishing a connection between the Eisenstein
series and the representation theory of $G$. Similarly, the functions
$g \mapsto E_i(g,s)_n$ defined in \equ{eq: twisted eisenstein 2} are eigenfunctions
for the Casimir operator on $G$.

Also recall our
normalization sending a cusp of $\Gamma$ 
to $\infty$ so that the stabilizer group becomes $N_0$. If $\Gamma$
has one cusp then this means
that $z \mapsto E(z,s)$ has a 
periodicity property $E(z,s) = E(u_1z, s) = E(z+1, s)$. We can exploit
this periodicity by 
developing $E(z,s) = E(x+\mathbf{i}y,s)$ in a Fourier series $\sum_m
a_m(y,s) e^{2\pi \mathbf{i}mx}.$ Furthermore, if $\Gamma$ has more
than one cusp and $i, j$ represent two of them, 
then $z \mapsto E_i(\mathfrak{s}_j z, s) = \sum_m
a_{i,j, m}(y,s) e^{2\pi \mathbf{i}mx}$ is also 1-periodic, and this leads
to interesting relations between the functions $a_{i,j, m}$.

We now turn to Selberg's results. By Corollary \ref{cor: convergence},
as an absolutely convergent series of holomorphic functions, the
functions $s \mapsto E_i(g, s)$ are holomorphic on $\{\re(s) >
1\}$. A fundamental issue is to extend the functions to the entire
plane, and here we have: 

\begin{theorem}[Selberg c. 1953]\name{thm: analytic continuation}
The functions $s \mapsto E_i(z,s)$ have a meromorphic continuation to the 
complex plane. There is a pole at $s=1$ with residue
$\frac{1}{\mathrm{covol}(\Gamma)}$, and all 
other poles with $\re(s) \geq \frac12$ are contained in
$\left(\frac12,1\right]$ (in particular there are no poles at $s =
1/2$).  All poles are simple. 
\end{theorem}

The second basic result is an analytic continuation
according to which one may recover the values of $E_i(g,\cdot)$ at $s$ from the
values at $1-s$. To state this we use the notation introduced after
Definition \ref{def: tempered}, and let 
$\mathbf{\Gamma}$ denote the classical $\Gamma$-function. For each $1 \leq
i,j \leq k$ let 
\eq{eq: constant term}{
\varphi_{ij} (s) = \sqrt{\pi} \, \frac{\mathbf{\Gamma}\left( s - \frac12
  \right)}{\mathbf{\Gamma}(s)}\sum |c|^{-2s}, \ \ \ \Phi(s) =
\left(\varphi_{ij}(s) \right)_{i,j=1}^k 
}
where the sum ranges over distinct representatives
$\left(\begin{matrix} a & b \\ c & d\end{matrix} \right)$ of double
cosets  
in $N'_0 \backslash \mathfrak{s}_i^{-1} \Gamma \mathfrak{s}_j/N'_0$ with $c \neq
0$.
The function $\varphi_{ij}$ has another definition in terms of the constant term
in the Fourier expansion of $z
\mapsto E_i(\mathfrak{s}_jz,s)$, see \cite[\S2.2]{Kubota}. 
The matrix  $\Phi(s)$ is sometimes called the {\em
  constant term matrix} corresponding to $\Gamma$, and sometimes
called the {\em scattering matrix}. The poles of $s
\mapsto E_i(z, s)$ with $\re(s) \geq 1/2$ are also poles of $\Phi$.  \combarak{Please
  check there are no mistakes in \equ{eq: constant term}. I took it from
  Kubota, but I needed to manipulate.}

\begin{theorem}[Selberg c. 1953]\name{thm: functional equation}
The matrix valued function $\Phi$ satisfies 
$$
\Phi(s)\Phi(1-s) =\mathrm{Id},
$$
and the column vector $\EE(z,s) = (E_1, \ldots, E_k)$ satisfies 
$$
\EE(z,s) = \Phi(s) \EE(z, 1-s).
$$
\end{theorem}

\subsection{Main term asymptotics and quadratic constant}\name{subsec:
  main term} 
As Veech noted, it is well-known to number-theorists that
the existence of a meromorphic continuation with a 
simple pole at $s=1$, already 
implies Theorem \ref{thm: Veech 89}, part (c). To see this, recall
the Wiener-Ikehara Tauberian theorem
(see e.g. \cite[Theorem 17]{Widder}), which was developed in
order to simplify proofs of the prime number theorem, and states: 

\medskip

{\em 
Suppose $\psi: \R_+ \to \R_+$ is monotone non-decreasing, $A \in
\R$, and suppose the integral $\int_0^\infty e^{-st} \psi(t)dt,$ where $s = \sigma +
\mathbf{i}\tau$, converges for $\sigma>1$ to a function $f(s)$ which
satisfies that 
$\lim_{\sigma \to 1+} \left( f(s) - \frac{A}{s-1} \right)$ exists, converges
uniformly, and
defines a uniformly bounded function in every interval $\tau \in [-a,
a]$, for all $a>0$. Then $\frac{\psi(t)}{e^t} \to_{t \to \infty} A.$ 
}

\medskip

To obtain part (c) of Theorem \ref{thm: Veech 89}, suppose $\Gamma v$
is a discrete orbit for a nonuniform lattice $\Gamma$ corresponding
the $i$-th cusp of $\Gamma$, and apply the
Wiener-Ikehara theorem with $A$ the residue of $E_i(g,s)$ at $s=1$, and
$\psi(t) = \mathbf{N}(g, e^{t/2})$. The 
hypotheses of the Wiener Ikehara theorem are justified by \equ{eq:
  integration by parts}, a change of variables $R = e^{t/2}$, and
Theorem \ref{thm: analytic continuation}. Here 
\eq{eq: def c0}{
A = \left\{ \begin{matrix} 
\frac{1}{\mathrm{covol}(\Gamma)} & \Gamma \text{ does not contain
  -}\mathrm{Id} \\  \frac{2}{ \mathrm{covol}(\Gamma)} & \text{otherwise} \end{matrix}
\right. }
will be the quadratic growth constant $c_{\Gamma v}$ 
provided $v$ satisfies $v = \mathfrak{s}_i \ebf_1$.  

\medskip 

{\em Warning (continued):} In \cite{Veech Eisenstein} the quadratic
growth constant is given as $\frac{1}{\mathrm{covol}(\Gamma)}$, but the
groups he considers do contain -$\mathrm{Id}$. The discrepancy is due
to the fact that Veech only counts closed cylinders and saddle
connections on surfaces, and each of these gives rise to two holonomy
vectors, depending on orientation. 

\medskip

Veech was not content with deriving Theorem \ref{thm: Veech 89}(c)
from known results about Eisenstein series. In 1998 
he reversed the logic, reproving the result using
ergodic-theoretic ideas introduced in \cite{EM}, and using this, 
obtained a continuation result for $E(g,s)$. Namely he showed that the
limit $\lim_{s\to 1} (s-1)E(g,s)$ exists 
along any sequence approaching $s=1$ nontangentially from $\{z \in \C:
\re(z)>1\}$, and used this to provide an alternative derivation of the
formula \equ{eq: def c0} 
for the quadratic growth constant. See \cite[\S
16]{Veech Siegel} for more details.  

\subsection{$\Theta$-transform}\name{subsec: theta transform} 
Let $\Gamma$ be a non-uniform lattice in $G$, and $\Gamma v $ a
discrete orbit in the plane. We will assume throughout this section
that $v$ corresponds to the $i$-th cusp of $\Gamma$ and is normalized
so that $ v = \mathfrak{s}_i \ebf_1$. Putting different weights on different
points on the plane amounts to choosing $f : \R^2 \to \C$, and
defining 
$$
\Theta_f : G/\Gamma \to \C, \ \ \Theta_f (g\Gamma ) = \sum_{u \in
  g\Gamma v} f(u). 
$$
We will refer to the map $f \mapsto \Theta_f$ as the {$\Theta$-transform.}
Note that this definition extends \equ{eq: eisenstein}, in that
$g \mapsto E^{(\Gamma, v)}(g,s) = \Theta_{f}(g)$ for $f(u) =
\|u\|^{-2s}.$ 
As before we need to worry about convergence issues, and we will
assume for the moment that $f$ has compact
support contained in $\R^2 \sm \{0\}$. Note that this is not satisfied
for \equ{eq: eisenstein} and it will make the $\Theta$-transforms we
consider easier to handle analytically. This will already be apparent
in the following proposition, in which we discuss the
$\Theta$-transform of smooth functions which have a special form. 

Write 
$h_R(x) = h\left(\frac{x}{R}\right)$. 
With this notation we have the following extension of
\equ{eq: eisenstein mellin}:    

\begin{proposition}\name{prop: Lemma 1.6}
Let $f: \R^2 \sm \{0\}\to \C$ be a smooth compactly supported
function, let $\rho: \R/2\pi \Z \to \R_+$ be smooth,
let $\psi: \R_+ \to \C$ be smooth and compactly supported, and let 
$\Psi=\mathcal{M}\psi $ be the Mellin transform of $\psi$ as in
\equ{eq: Mellin transform}. Let $\sigma>2$ and denote by $\Theta$ the
transform associated with the 
orbit $\Gamma v$ corresponding to the $i$-th cusp of
$\Gamma$, normalized so that $v = \mathfrak{s}_i \ebf_1$. 
 Then:
\begin{enumerate}
\item 
If $f (u) = \psi(\|u\|)$ is purely radial,
then 
\eq{eq: eisenstein mellin 3}{
\Theta_{f_R} (g\Gamma) = \frac{1}{2\pi \mathbf{i}} \int_{\re(s) = \sigma}
\Psi(s) E_i\left(g, \frac{s}{2} \right) R^s ds.
}

\item
Suppose 
$f\left(re^{\mathbf{i}\theta} \right) = \psi \left(\frac{r}{\rho(\theta)}
\right)$, 
and $\rho(\theta)^s = \sum_{n \in \Z} \hat{\rho}_{n}(s)
e^{\mathbf{i}n\theta}$ is the Fourier expansion of $\rho^s$. Then 
\eq{eq: twisted eisenstein mellin 4}{
\Theta_{f_{R}}(g\Gamma) = \frac{1}{2\pi \mathbf{i}} \sum_{n\in \Z}
\int_{\re(s)=\sigma} \Psi(s) \hat{\rho}_{-n}(s)  E_i\left( g,
  \frac{s}{2} \right)_{n} R^s ds.
}
\item
Suppose $f$ splits into angular and radial parts as 
$f\left(re^{\mathbf{i}\theta} \right)
= \psi(r) \rho(\theta),$ and let $\rho(\theta) = \sum_n \hat{\rho}_n
e^{\mathbf{i} n \theta}$ be the Fourier expansion of $\rho$. Then 
\eq{eq: twisted eisenstein mellin}{
\Theta_{f_R} (g \Gamma) = \frac{1}{2\pi \mathbf{i}} \sum_{n \in \Z} \hat{\rho}_{-n}
\int_{\re(s)=\sigma} \Psi(s) E_i\left(g, \frac{s}{2} \right)_{n} R^s ds.
}

\end{enumerate}
\end{proposition}

\begin{proof}
There is no need to prove (1) since it is the special case of (2) with
$\rho(\theta) \equiv 1$. 
We will write the Fourier expansion $\rho(\theta)^s = \sum_n \hat{\rho}_{n}(s)
e^{\mathbf{i} n \theta}$ as $ \sum_n \hat{\rho}_{-n}(s)
e^{-\mathbf{i} n \theta}$. The Fourier series converges absolutely for
each $s$ since $\rho$ is smooth, and the coefficients admit an upper
bound 
\eq{eq: upper bound fourier}{
|\hat{\rho}_{n}(s)| \leq 2\pi \, \|\rho\|_\infty^{\sigma}, 
\ \ \text{where } \sigma =\re(s).
}
More generally, applying integration by parts twice, we see that 
\eq{eq: upper bound fourier 2}{
|\hat{\rho}_{n}(s)| \ll \frac{|s|^2}{n^2},
}
where the implicit constant depends on $\sigma, \, \|\rho\|_\infty, \, 
\|\rho'\|_\infty,$ and $\|\rho''\|_\infty$.
The Mellin transform $\mathcal{M}\psi$
satisfies $(\mathcal{M}\psi_R )(s)= R^{s} (\mathcal{M}\psi) (s)$ and
so by Mellin inversion
$$
\psi_R\left(\frac{y}{\rho}\right) = \frac{1}{2\pi \mathbf{i}} \int_{\re(s) = \sigma} \Psi(s)
y^{-s} R^{s} \rho^s ds.
$$
Plugging this into the
definition of  $\Theta_{f_R}$ and  writing each $u$ as $
\|u\|e^{\mathbf{i}\theta_u}$ we obtain 
\[
\begin{split}
\Theta_{f_R}(g\Gamma) = & \sum_{u \in g\Gamma v} \psi_R\left(\frac{\|u\|}{\rho (
  \theta_u)} \right)  = \frac{1}{2\pi \mathbf{i}} \int_{\re(s) = 
\sigma} \Psi(s) R^s \, \left(\sum_{u \in g\Gamma v} \|u\|^{-s}
\rho(\theta_u)^s \right) ds \\
= & \ \frac{1}{2\pi \mathbf{i}} \int_{\re(s) = \sigma}
\Psi(s) R^{s} \sum_{n \in \Z} \hat{\rho}_{-n}(s) \left( \sum_{u \in g\Gamma v}
    \|u\|^{-s} e^{-\mathbf{i} n 
    \theta_u} \right) ds \\
= & \ \frac{1}{2\pi
  \mathbf{i}} \sum_{n \in \Z} \int_{\re(s) = \sigma} 
\Psi(s) \hat{\rho}_{-n}(s) R^{s} \left( \sum_{u \in g\Gamma v}
    \|u\|^{-s} e^{-\mathbf{i} n 
    \theta_u} \right) ds. 
\end{split}
\]
To justify switching order of integration and summation in the first
line use the quadratic growth of the set $g\Gamma v$ (Proposition
\ref{prop: quadratic upper 
  bound}) and the assumption $\sigma >2$. In the second line, use also
\equ{eq: upper bound fourier}, and in the third line use Proposition \ref{prop: quadratic upper
  bound}, \equ{eq: upper bound fourier 2},
and the fact that $\Psi$ decays faster than any polynomial along the
line $\re(s)=\sigma$. 
Formula \equ{eq: twisted eisenstein mellin 4} now follows by plugging
in \equ{eq: twisted eisenstein 2}. 

For (3), we have 
\[
\begin{split}
\Theta_{f_R}(g\Gamma) = & \sum_{u \in g\Gamma v} \rho (
  \theta_u) \psi_R\left(\|u\|\right)  \\
= &\sum_{u \in g\Gamma v} \left( \sum_{n
    \in \Z} \hat{\rho}_{-n} e^{-\mathbf{i}n\theta_u} \right) \, 
  \frac{1}{2\pi \mathbf{i}} \int_{\re(s) =  
\sigma} \Psi(s) R^s \|u\|^{-s}
 ds \\
= &  \frac{1}{2\pi \mathbf{i}} \sum_{n \in \Z} \hat{\rho}_{-n} \int_{\re(s) = \sigma}
\Psi(s) R^{s} \left( \sum_{u \in g\Gamma v} 
    \|u\|^{-s} e^{-\mathbf{i} n 
    \theta_u} \right) ds, 
\end{split}
\]
and again we plug in \equ{eq: twisted eisenstein 2}, leaving it to the
reader to justify changing the order of sums and 
integrals.
\end{proof}

\subsection{Additional properties}
We will need the extensions of the results of \S \ref{subsec: Selberg
  results} to twisted Eisenstein 
series, and also some further properties. For convenience we
collect all the results we will need, including results already discussed
above, in the following list. 

\begin{theorem}\name{thm: factsheet}
Let $\Gamma$ be a nonuniform lattice in $G$ with $k$ cusps. Let
$\mathfrak{s}_i$ be the elements conjugating these cusps to $\infty$ as in
the discussion preceding \equ{eq: eisenstein 2}. Let $E_i(z,s)$
(resp. $E_i(z,s)_{2n}$) denote the (twisted) Eisenstein series as in
\equ{eq: eisenstein 2} (resp. \equ{eq: twisted eisenstein}). 
Then there is a function $\omega : \R \to \R$ (see \equ{eq: def omega} for an
explicit definition) such that the following hold: 
%
\begin{itemize}
\item[{\bf (AC)}]
The functions $E_i(z,s)_{2n}$ are absolutely convergent for
$\re(s)>1$, and for any $\zeta>0$, they are uniformly bounded and 
uniformly convergent on sets of the form $\{\re(s) \geq 1+\zeta\}$.

\item[{\bf (M)}]
The functions $s \mapsto E_i(z,s)_{2n}$ have a meromorphic continuation to all of
$\C$. 
\item[{\bf (P)}]
The poles of $s \mapsto E_i(z,s)_{2n}$ with $\re(s) \geq 1/2$ are all
simple, lie in $(1/2,1]$, and are contained in the set $(s_\ell)$ of
poles of the constant term matrix $\Phi$ of \equ{eq: constant term}. 
\item[{\bf (1)}]
There is a pole at $s=1$ if and only if $n=0$, with residue
$\mathrm{covol}(\Gamma)^{-1}$. 

\item[{\bf (1/2)}]
The functions 
$E_i(z,s)_{2n}$ have no poles on the line $\re(s)=1/2$.

\item[{\bf $(\omega)$}]
For all $t \in \R$, $\omega(t) \geq 1, \, \omega(-t) = \omega(t)$ and
for $T \geq 1$,
$\int_{-T}^T
\omega(t) dt \ll T^2$, where implicit constants depend on
$\Gamma$. 
\item[{\bf (G1)}] If $n \in \Z$, 
 $\re(s) \geq 1/2$ 
and $|t| \geq |n|+1$ then $E_{i}(z,s)_{2n} \ll |t|\,
\sqrt{\omega(t)}$, where  $t = \im(s)$ 
(implicit constants depend on $z$ and
$\Gamma$ but not on $n$ or $\re(s)$). 
\item[{\bf (G1/2)}] For all $n$, 
$\int_{-T}^T \left|E_{i}\left(z,\frac12+\mathbf{i}t \right)_{2n}
  \right|^2dt \ll (T+|n|)^2$ 
(implicit constants depend on $\Gamma$). 
\end{itemize}
\end{theorem}

\begin{proof}
For $n=0$, all items are given in
\cite{Selberg gottingen}, see also \cite{Hejhal,
  Kubota}. The extension of the first five properties to general $n$
is given in \cite[Chapter 6]{Kubota} (see also \cite{Sarnak
  horocycles}). Property  {\bf (G1/2)} is extended to arbitrary $n$ by 
Marklof and Str\"ombergsson in \cite{MS} (in \cite{MS} only the case
of the integral over $[0,T]$ is discussed, but the proof extends
verbatim to the interval $[-T, 0]$). To the best of our 
knowledge, there is no presentation of property {\bf (G1)} for general
$n$ in the literature. We fill this gap in the appendix to this paper,
see Theorem \ref{prop: in appendix}.
\end{proof}
\section{A bound $O\left(R^{\frac{4}{3}}\right)$}\name{sec: expository}
The following is the main result of this section. It immediately
implies Theorem \ref{thm: good}. 

\begin{theorem}\name{thm: more detailed good}
Suppose $\Gamma$ is a lattice in $G$,
$\Gamma v$ is a discrete orbit 
corresponding to the $i$-th cusp of $\Gamma$, $E_i(g, s)$ is the
corresponding Eisenstein series, and $s_0=1 > s_1 > \cdots > s_r >
1/2$ are the poles 
of $E_i(g, \cdot)$. Then there are $c_0, \ldots, c_r$ such that 
$$
\mathbf{N}(g, R) = c_0 R^2 + \sum_{\ell =1}^r  c_\ell R^{2s_\ell} +
O\left(R^{\frac43 
}\right).
$$
Furthermore, if $v$ is rescaled so that $\mathfrak{s}_i \ebf_1 = v$, then the $c_\ell$ are the
residues of $s \mapsto \frac{E_i(g, s)}{s}$ at the
poles $s_\ell.$ In particular,  
the quadratic growth constant $c_0$ is given by formula \equ{eq: def
  c0}. 
\end{theorem}

The basic idea for  the proof of Theorem \ref{thm: more detailed good}
is a `contour shift' argument, as follows. We recall \equ{eq:
  eisenstein mellin} and \equ{eq: 
  eisenstein mellin 2}, which imply
that for $\sigma >1$, for a large parameter $T$, $\mathbf{N}(g, R) \approx 
\frac{1}{2\pi} \int_{-T}^T E(z, \sigma + \mathbf{i}t) \frac{R^{2(\sigma
    + \mathbf{i}t)}}{\sigma + \mathbf{i}t } dt.$ This is a path
integral over the line segment 
$L_{\sigma, T} = \{\sigma + \mathbf{i}t : t \in [-T, T]\}$ introduced  in \S
\ref{subsec: some sums}.  Since $s \mapsto E(z, s)$ is meromorphic in
all of $\C$, the Cauchy
residue formula makes it possible to replace this path integral over $L_{\sigma, T}$
with a path integral over $L_{1/2, T}$ and the two horizontal segments
$H^{\pm}=\{s \pm \mathbf{i} T : s \in [1/2, \sigma]\}$, taking into account the
residues in the rectangle bounded by these segments. We need to show
that the contribution of the integral over 
the segments $H^{\pm}$ is negligible,
compute the contribution of the poles in the Cauchy formula, 
and evaluate the integral over $L_{1/2, T}$. Each of these steps
presents difficulties as stated. To bypass them we recall that if
$\chi = \chi_{[0,1]}$ denotes the 
indicator function of $[0,1]$ and $f (u) = \chi(\|u\|)$ then
$\mathbf{N}(g, R)= \Theta_{f_R}(g\Gamma)$. We can 
justify the contour shift argument if $f$ is replaced by a smooth
compactly supported 
approximation $f^{(U)}$ (where $U$ is an approximation parameter), and
in this way, obtain bounds on the growth of 
$\Theta_{f^{(U)}_R}(g\Gamma)$ as $R \to \infty$. To make use of this
we bound the difference $\left|\mathbf{N}(g, R) -\Theta_{f^{(U)}_R}(g\Gamma) \right|$ as
well as the differences in the residues of the sums for $f$ and
$f^{(U)}$, and optimize the choice of $U$ as a function of $R$ to make
the combined error as small as possible. 

In order to justify the contour shift we will need the following:

\begin{proposition}\name{prop: for contour shift}
Let $E$ be a meromorphic function on $\C$, let $a<b$, and let $\Psi$ be
a holomorphic function defined in a neighborhood of $\{s \in \C: a \leq \re(s)
\leq b\}$, such that:

\begin{enumerate}
\item[(i)] $E$ has finitely many poles $(s_\ell)$ with $\re(s) \in [a,b]$. They
  are all simple poles, all on the real line, and there are no poles
  at $s=a$ and $s=b$. 
\item[(ii)]
There is a function $\omega : \R \to \R$ satisfying the conclusions
of Theorem \ref{thm: factsheet} item {\bf $(\omega)$}, and such
that for all $|t| \geq 1$,  $|E(s)|\ll 
\sqrt{\omega(t)} \, |t|$ (where $t = \im(s)$). 
\item[(iii)]
For any $k>0$ there is $C'$ such that for all $\sigma \in [a,b]$, 
$$|\Psi(\sigma + \mathbf{i}t)| \leq \frac{C'}{t^k}.
$$
\item[(iv)] 
For $\sigma =a$ and $\sigma =b$, 
the integrals 
\eq{eq: first v}{\int_{-\infty}^\infty 
E\left(\sigma  + \mathbf{i}t  \right) \, \Psi\left(\sigma +
\mathbf{i}t  \right) \,dt }
converge absolutely.
\end{enumerate}
Then 
\eq{eq: what we want shift}{
\frac{1}{2\pi \mathbf{i}} \int_{\re(s) = b} E(s)\Psi(s) ds =
\frac{1}{2\pi \mathbf{i}}\int_{\re(s) = a} E(s)\Psi(s) ds 
    + \sum_\ell \Psi(s_\ell) \mathrm{Res}|_{s=s_\ell} (E)
}
(where $\mathrm{Res}|_{s=s_0}h(s)=\lim_{s\to s_0}(s-s_0)h(s)$).  

\end{proposition}

\ignore{
For the proof of Proposition \ref{prop: for contour shift} we will
need the following well-known fact: 
\begin{proposition}[Phragm\'en-Lindel\"of principle, see
  e.g. \cite{bergeron}]\name{prop: phrag lind} 
Let $a < b$ and $t_0$ be real numbers, let $\Omega$ be the `half-strip' $\{z \in \C:
\re(z) \in [a,b], \im(z) > t_0\}$, let $\alpha, \beta, C$ be positive,
and let $f: \Omega \to \C$ 
be  a holomorphic function satisfying the following bounds:
\begin{enumerate}\item[(I)]
There is some $\alpha>0$ such that for all $z \in \Omega$, 
$$
|f(z)| \leq Ce^{|z|^\alpha}.
$$
\item[(II)]
For all $t >t_0$, and for both $\sigma = a$ and $\sigma =b$, 
$|f(\sigma+ \mathbf{i}t)| \leq C|1+t|^\beta$.

\end{enumerate}
Then there is $C'$ such that for all $\sigma \in [a,b]$, $|f(\sigma+
\mathbf{i}t)| \leq C'|1+t|^\beta$.

\end{proposition}

\begin{proof}[Proof of Proposition \ref{prop: for contour shift}]
}
\begin{proof}
Let $\tau>1$, and consider the integral of $E(s) \, \Psi(s)$ on
the rectangle 
$$R_\tau = \{s \in \C: \re(s) \in [a,b], \im(s) \in [\tau,
\tau+1]\}.$$ 
We have  
\[\begin{split}
\int_{\tau}^{\tau +1} \left| \int_a^b E(\sigma + \mathbf{i}t) \Psi (\sigma +
  \mathbf{i}t) \, d\sigma \right| dt& \stackrel{(iii)}{\ll}
\frac{1}{\tau^3}  \int_{\tau}^{\tau +1}\int_a^b
|E(\sigma + \mathbf{i} y)| \, dy \, d\sigma \\
& \stackrel{(ii)}{\ll} \frac{1}{\tau^3} \int_{-(\tau+1)}^{\tau+1} \sqrt{\omega(y)} \,
\, y \,  dy \\ 
& \stackrel{\text{Cauchy-Schwarz}}{\ll}
\frac{1}{\tau^3}\left(\int_{-(\tau+1)}^{\tau+1} \omega(y) dy
\right)^{1/2} 
\tau^{3/2} \\ 
& \stackrel{\mathbf{(\omega)}}{\ll} \tau^{-1/2} \longrightarrow_{\tau
 \to \infty} 0. 
\end{split}\]

Hence for each $n \in \N$ there is $\tau_n \in [n,
n+1]$ such that 
\eq{eq: the same is true}{
\left| \int_a^b E(\sigma + \mathbf{i}\tau_n) \Psi (\sigma +
  \mathbf{i}\tau_n) \, d\sigma \right| \to_{n \to \infty} 0.
}
By the same argument, there are $\tau_{-n}
\in [-(n+1),-n]$ such that \equ{eq: the same is true} also holds with
$\tau_{-n}$ instead of $\tau_n$. 

\ignore{
Let $r>0$ be smaller than the minimal distance between poles and
between poles and the 
points $a,b$, and define the auxilliary function 
$$
f(s) = f^{(r)}(s) = \frac{1}{r} \int_{-r/2}^{r/2} E(s+\mathbf{i}y)dy. 
$$
Then $f$ is holomorphic in a neighborhood of $\Omega =
\{z \in \C : \re(z) 
\in [a,b], \, \im(z) \geq r\},$ and by the mean value theorem we have 
\eq{eq: mvt}{
\lim_{r\to 0} f^{(r)}(s) = E(s)  \ \ \ \ (\forall s \in \Omega).
}
We also have 
\[\begin{split}
|f(s)| & \ll
\int_{r/2}^{t+r} |E(\sigma + \mathbf{i}y)| \, dy
\stackrel{(ii)}{\ll} \int_{r/2}^{t+r} \sqrt{\omega(y)} \,
\, y \,  dy \\ & \ll \left(\int_{-(t+r)}^{t+r} \omega(y) dy \right)^{1/2}
t^{3/2} \stackrel{\mathbf{(\omega)}}{\ll} t^{5/2} \end{split}\]
(where $t = \im(s) \geq r$, and implicit constants depend on $r$). 
A similar argument gives the same bound for all $t \leq -r$. Now
define 
$$h(s) = f(s) \Psi(s)$$
so that again $h$ depends on $r$, and by (v), 
\eq{eq: h decays}{
\lim_{\rho \to \pm \infty} h(\sigma + \mathbf{i}\rho) =0, \ \text{
  uniformly for } \sigma \in [a,b]. 
}
}
Since $\Psi$ is holomorphic, and $E$ holomorphic outside a set of finitely many poles, we can now apply the Cauchy residue formula for the contour integral of
$E(s)\Psi(s)$ over the boundary of the rectangle 
$$
\{s \in \C: \re(s) \in [a,b], \im(s) \in [\tau_{-n}, \tau_n]\}.
$$
The intergrals along the horizontal boundaries $[a,b] \times
\{\tau_{\pm n} \}$ go to 0 as $n\to \infty$, and the integrals along
the vertical boundaries $\{a, b \} \times [\tau_{-n}, \tau_n]$ tend to
the integrals in \equ{eq: first v}. The result follows. 
\ignore{

for a large parameter $\rho_2$, and a small parameter $\rho_1 > r$, compute the path integral of $h$
along 
the sides of the rectangle
bounded by the four lines $\{\re(z) =a\}, 
\{\re(z) = b\}, \{\im(z) = \rho_1\}, \{\im(z) = \rho_2\}$ (with the
natural orientation). The
integral along the side $\{\im(z) = \rho_2\}$ tends to zero as $\rho_2 \to
\infty$ by \equ{eq: h decays}.  
By Cauchy's theorem, since $h$ is holomorphic in $\Omega$, 
we have 
\eq{eq: from cauchy}{
\int_{\rho_1}^\infty h(b+ \mathbf{i}y) dy =
\int_{\rho_1}^\infty h(a + \mathbf{i}y) dy
    - \int_a^b h(x + \mathbf{i}\rho_1)dx 
}
and this identity holds for all values of $\rho_1 >r>0$. A similar
formula holds for the corresponding rays with negative imaginary
part. Therefore 
\eq{eq: almost there}{
\int_{\substack{\re(s) = b \\ |\im(s)| \geq \rho_1}} h(s)
ds - \int_{\substack{\re(s) = a \\ |\im(s)| \geq \rho_1}} h(s)
ds 
= 
\int_{\substack{\im(s)  =
-\rho_1 \\ \re(s) \in 
  [a,b]}} h(s)ds -
 \int_{\substack{\im(s)  =
\rho_1 \\ \re(s) \in 
  [a,b]}} h(s)ds.
}
Recall that $h = h^{(r)}$ depends on $r$. We will take $r \to 0$ in
both sides of \equ{eq: almost there}. First, 
using \equ{eq: mvt} and the dominated
convergence theorem, we see that 
\[
\lim_{r \to 0} 
 \int_{\substack{\im(s)  =
\pm \rho_1 \\ \re(s) \in 
  [a,b]}} h^{(r)}(s)ds
= 
 \int_{\substack{\im(s)  =
\pm \rho_1 \\ \re(s) \in 
  [a,b]}} E(s) \Psi(s)ds.
\]
Also, using assumption \equ{eq: second v} and Fubini's theorem, we see
that 
\[
\int_{\substack{\re(s) = b \\ |\im(s)| \geq \rho_1}} h^{(r)}(s)
ds  = \frac{1}{r} \int_{-r/2}^{r/2} \int_{\substack{\re(s) = b \\ |\im(s)| \geq \rho_1}}
E(s +\mathbf{i}y) \Psi(s) ds \, dy 
\]
\combarak{there is a problem here...}
 we see that
\equ{eq: almost there} holds with $E(s)\Psi(s)$ instead of $h(s)$. On
the other hand, by the Cauchy residue formula, if we set 
$$
I(\rho, x) = \int_{-\rho}^{\rho} E(x + \mathbf{i}y) \Psi (x+ \mathbf{i}y) dy
$$
then
\[\begin{split}
& \int_{\re(s)  = b, |\im(s)|\geq \rho_1} E(s) \Psi(s) ds -
 \int_{\re(s)  = a, |\im(s)|\geq \rho_1}
E(s) \Psi(s) ds 
\\ = &   2\pi \mathbf{i} \sum_{\ell} \mathrm{Res}|_{s=s_\ell} E(s)
\Psi(s) + I(\rho_1, a) - I(\rho_1, b) .
\end{split}\]
Using (i) we see that the terms $I(\rho_1,
a)$ and $I(\rho_1, b)$ tend to 0 as $\rho_1 \to 0$. Equation \equ{eq:
  what we want shift} follows. 
%
}
\end{proof}
\medskip

\begin{proof}[Proof of Theorem \ref{thm: more detailed good}]
Let $\beta : \R \to [0,1]$ be a smooth function satisfying 
$$
\beta(x) = \left\{ \begin{matrix}  0 & \text{ for }  x \leq  0.1 \\ 1 & \text{ for } x \geq 1 
\end{matrix} \right.
$$
and for a parameter $U>2$ let 
$$
\psi^{-(U)} = \left\{\begin{matrix} \beta(Ux) & x \leq 1/2 \\ 
\beta(U(1-x)) & x \geq 1/2\end{matrix}  \right.  \ \ \ \ \ \text{ and }
\ \ 
\psi^{+(U)} = \left\{\begin{matrix} \beta(Ux) & x \leq 1/2 \\  
\beta(1+U(1-x)) & x \geq 1/2\end{matrix}  \right. .
$$
Let $\chi$ denote the indicator function of 
$[0,1]$. Our choices imply that $\psi^{\pm(U)}$ are supported on a
compact 
subset of the positive real line and satisfy 
\eq{eq: approximation function properties}{
\begin{split}
x \geq \frac{1}{U} \ \ & \implies \ \ \psi^{-(U)} (x) \leq \chi(x)
\leq \psi^{+(U)}(x), \\ 
x \notin \left[0, \frac{1}{U} \right] \cup \left[1-\frac{1}{U},
  1+\frac{1}{U}\right] \ \ & \implies \ \ 
\psi^{-(U)}(x) = \chi(x)= \psi^{+(U)}(x), \\
\sup_{x \in \R} \frac{d^{\ell}\psi^{\pm(U)}}{dx^{\ell}} \  & \ \ = \ \ \ \
O\left(U^{\ell} \right).
\end{split}
}
We will define $f^{\pm}_{R, U}(u) = \psi^{\pm(U)}
\left(\frac{\|u\|}{R} \right )$. 
Then by \equ{eq: approximation function properties} we have 
\eq{eq: above and below}{
\Theta_{f^-_{R,U}}(g\Gamma)  \leq \mathbf{N}(g,R) \leq
\Theta_{f^+_{R, U}}(g\Gamma)+ \mathbf{N}\left(g, \frac{R}{U}\right).
}
We will obtain bounds for $\Theta_{f^\pm_{R,U}}(g\Gamma)$ using
\equ{eq: eisenstein mellin 3} and a contour shift argument, and then
combine this with \equ{eq: above and below} and optimize the choice of
$U = U(R)$ to obtain good bounds for $\mathbf{N}(g,R)$. To simplify notation we
omit the superscript $\pm$ from now, that is $f_{R,U}$ stands for any
one of $f^{\pm}_{R,U}$ and $\Psi^{(U)}$ stands for the Mellin
transform of any one of the 
$\psi^{\pm(U)}$. 

\medskip 

{\em Step 1. Dependence of the residues on the approximation
  parameter.}
Let $c'_\ell$ be the residue of $s \mapsto E_i(g,s)$ at $s=s_\ell$, let $c_\ell$ be the
residue of $s \mapsto \frac{E_i(g,s)}{s}$ at $s=s_\ell$ and let $c_\ell(U)$ be the
residue of $s \mapsto \Psi^{(U)}(s) E_i\left(g, \frac{s}{2}\right)$ at 
$s = 2s_\ell$. The $c'_\ell$ are nonzero by
{\bf (P)} and the $s_\ell$ satisfy $s_\ell > 1/2$, and we have
$c_\ell=\frac{c'_\ell}{s_\ell}$ and 
$c_\ell(U)= 2 c'_\ell \, \Psi^{(U)}(2s_\ell).$ By \equ{eq: approximation
  function properties}, we have
\[
\begin{split}
\left| \frac{1}{2s_\ell} - \Psi^{(U)}(2s_\ell)\right | & \leq \left| \int_0^1
  y^{2s_\ell-1} dy - \int_0^{\infty} \psi^{(U)}(y) y^{2s_\ell-1} dy\right |
\\
& \leq 
\int_0^\infty \left | \psi^{(U)}(y) -\chi(y) \right| y^{2s_\ell-1} 
dy \\
& \stackrel{\equ{eq: approximation function properties}}{\leq} 2
\left( \int_{0}^{1/U} y^{2s_\ell-1} dy + \int_{1-1/U}^{1+1/U} 
y^{2s_\ell-1} dy\right) = O\left( \frac{1}{U}\right),
\end{split}
\]
and thus 
\eq{eq: residue difference}{
c_\ell - c_\ell(U) = 2c'_\ell \left( \frac{1}{2s_\ell} - \Psi^{(U)}(2s_\ell)
\right) =O \left(
  \frac{1}{U} \right).}

\medskip 

{\em Step 2. Bounding the integral over $\re(s)=1$.} 
We will bound $\int_{\re(s) = 1} \Psi^{(U)}(s) E_i
\left(g, \frac{s}{2}\right) R^s ds$ in terms of $R$ and $U$, and to
this end we will bound  
$$
I = \int_1^\infty \Psi^{(U)}(1 + \mathbf{i}t) E_i\left (g, \frac{1 +
    \mathbf{i}t}{2} \right) R^{1 + \mathbf{i}t} dt.  
$$

We first prove that for all 
$U \geq 2$, 
\eq{eq: lemma 1.13}{
|t|\geq U \ \implies \ 
\Psi^{(U)}(1 + \mathbf{i}t) =
O\left(
\frac{U
}{|t|^{2
}} \right)
} 
and 
\eq{eq: lemma 1.13.1}{
1 \leq |t|\leq U \ \implies \ \Psi^{(U)}(1 + \mathbf{i}t) =
O\left(
\frac{1}{|t|} \right).
}
Moreover we will establish such a bound for $\sigma+ \mathbf{i}t$ in
place of $1 + 
\mathbf{i}t$, where the implicit constant is uniform as
long as $\sigma$ varies in a
closed interval of positive reals. 
To see \equ{eq: lemma 1.13}, apply  integration by parts twice, and use that $\psi^{(U)}$ and all its
derivatives vanish for $y \notin \left(0, \frac{1}{U} \right] \cup
\left[1-\frac{1}{U}, 1+\frac{1}{U} \right]$, to
obtain for $s = \sigma+\mathbf{i}t$:
\[\begin{split}
\left| \Psi^{(U)}(\sigma+ \mathbf{i}t)\right| & = \left| \int_0^\infty
  \left[\psi^{(U)} \right]'(y)\frac{y^s}{s}
dy \right|  =  \left| \int_0^\infty \left[\psi^{(U)} \right]''(y)\frac{y^{s+1}}{s(s+1)}
dy \right | \\ &  =  \left| \int_0^{1/U} \left [\psi^{(U)}\right]''(y)\frac{y^{s+1}}{s(s+1)}
dy + \int_{1-1/U}^{1+1/U} \left[\psi^{(U)} \right ]''(y)\frac{y^{s+1}}{s(s+1)}
dy\right |  \\ 
& =2^{\sigma+1} \, O \left(U^2\right) \, O \left(\frac{1}{U|t(t+1)|} \right) =
O\left(\frac{U}{|t|^2} \right),
\end{split} 
\]
proving \equ{eq: lemma 1.13}. 

Now if $1 \leq t \leq U$ and $U \geq 2$ then 
\begin{equation}\label{eq:tlessU}
\begin{split}
\left| \Psi^{(U)}(\sigma+ \mathbf{i}t)\right| & = \left| \int_0^\infty
  \left[\psi^{(U)} \right]'(y)\frac{y^s}{s}
dy \right|  \\ & =  \left| \int_0^{1/U} \left [\psi^{(U)}\right]'(y)\frac{y^{s}}{s}
dy + \int_{1-1/U}^{1+1/U} \left[\psi^{(U)} \right ]'(y)\frac{y^{s}}{s}
dy\right |\\ 
& =2^{\sigma +1}\, O(U) \, O\left( \frac{1}{U}\right) \, O\left( \frac{1}{\sqrt{\sigma^2+t^2}}
\right)  = O\left( \frac{1}{t}
\right),
\end{split} 
\end{equation}
proving \equ{eq: lemma 1.13.1}.


Writing $E(t)$ for $E_i\left(g, \frac{1 + \mathbf{i}t }{2} \right)$
and $\Psi(t)$ for $\Psi^{(U)}(1+\mathbf{i}t)$, 
this leads to 
\eq{eq: leads to}{
|I| \leq  \int_1^{\infty} \left|\Psi(t) E(t) R^{1+\mathbf{i}t} \right |dt
   \ll R 
\, \left[ \int_1^U 
\left|\Psi(t) E(t) \right|
dt +  \int_U^\infty  \left| \Psi(t) E(t) \right|
dt \right]
}
and we bound each of these integrals separately. 
By assumption {\bf (G1/2)} and Cauchy-Schwarz, for $1 \leq A \leq B$, 
\eq{eq: CS}{
\int_A^B |E(t)|dt \leq \sqrt{\int_A^B |E(t)|^2 dt} \, \sqrt{\int_A^B 
1 \, dt } \ll B^{\frac{3}{2}}.
}
Hence, by a dyadic decomposition, 
\[
\begin{split}
\int_1^U \left| E(t) \right| \, \left| \Psi(t) \right| dt &
\stackrel{\equ{eq: lemma 1.13.1}}{\ll}
\int_1^U \frac{|E(t)|}{t} dt \\
& \leq \sum_{0 \leq k \ll \log U} \frac{1}{U/2^{k+1}}
\int_{U/2^{k+1}}^{U/2^k} |E(t)|dt \\
& \stackrel{\equ{eq: CS}}{\ll} \sum_{0 \leq k \ll \log U}
\frac{1}{U/2^{k}} \left(\frac{U}{2^k} \right)^{\frac{3}{2}} \\
& \ll U^{\frac{1}{2}} \sum_{k \geq 0} 2^{-\frac{k}{2}} \ll U^{\frac{1}{2}}.
\end{split}
\]
Similarly, in the range $t \geq U$ we have 
\[
\begin{split}
\int_U^{\infty} \left| E(t) \right| \, \left| \Psi(t) \right| dt &
\stackrel{\equ{eq: lemma 1.13}}{\ll}
\int_U^{\infty} |E(t)| \, \frac{U}{t^2} dt \\
& \ll U \sum_{k \geq 0} \frac{1}{(2^{k}U)^2}
\int_{2^{k}U}^{2^{k+1}U} |E(t)|dt \\
& \stackrel{\equ{eq: CS}}{\ll} U \sum_{k\geq 0}
\frac{1}{(2^{k}U)^2} \left(2^{k+1}U\right)^{\frac{3}{2}} \\
& \ll U^{\frac{1}{2}} \sum_{k \geq 0} 2^{-\frac{k}{2}} \ll U^{\frac{1}{2}}.
\end{split}
\]
\ignore{

We have  
\eq{eq: bound on integral}{
\begin{split}
\int_U^\infty  \left| E(t) \right|
\frac{U^{\frac{1}{2}+\vre}}{t^{\frac32+\vre}} dt 
 & 
=U^{\frac12+\vre} \int_U^\infty
\frac{|E(t)|}{t^{1+\frac{\vre}{2}}} \, \frac{1}{t^{\frac12+\frac{\vre}{2}}} dt 
\\ & \leq
 U^{\frac12+\vre} \left(
  \int_U^{\infty} \frac{|E(t)|^2}{t^{2+\vre}} dt
\right)^{\frac12} \left( \int_U^{\infty} \frac{1}{t^{1+\vre}} dt
\right)^{\frac12}  \\ & \ll U^{\frac12+\vre} \left(
  \int_U^{\infty} \frac{|E(t)|^2}{t^{2+\vre}} dt
\right)^{\frac12},
\end{split}
}
where in the next to last line we used Cauchy-Schwarz, and implicit constants
depend on $\vre$.

Set $H(T) =\int_1^T \left| E(t)\right|^2 dt $ and note that 
\eq{eq: bound on H}{
H(T) = O\left(T^2 \right)
}
by {\bf (G1/2)}. Using integration by parts we have 
$$
\int_U^{\infty} \frac{|E(t)|^2}{t^{2+\vre}} dt = \lim_{T \to \infty} 
\left. \frac{H(t)}{t^{2+\vre}} \right|_{U}^{T} + \lim_{T \to \infty} (2+\vre) \int_U^T
\frac{H(t)}{t^{3+\vre}}dt. 
$$
Both of these terms are finite by \equ{eq: bound on H}, and 
thus \equ{eq: bound on integral} leads to 
$\int_U^\infty  \left| E(t) \right|
\frac{U^{\frac12+\vre}}{t^{\frac32+\vre}}  dt  =  O\left(U^{\frac12
    + \vre} \right).$ 
For the other integral in \equ{eq: leads to}, we have by
Cauchy-Schwarz, integration by parts and \equ{eq: bound on H} that
\[\begin{split}
 \int_1^U \frac{|E(t)|}{t} dt & \leq  U^{\frac12} \left(\int_1^U
  \frac{|E(t)|^2}{t^2} dt \right)^{\frac12} \\ & = U^{\frac12} \, \left(
  \left. \frac{H(t)}{t^2} \right|_{t=1}^{t=U} + 2 \int_1^U
  \frac{H(t)}{t^3} dt \right)^{\frac12} = O\left( \left(U \log U
  \right)^{\frac12} \right) = O \left( U^{\frac{1}{2} + \vre} \right).
\end{split}\]

}
Putting these estimates together we obtain 
$$
I = O \left(R \, U^{\frac12} \right). 
$$
The bound on the ray $\{1 + \mathbf{i} t : t
\leq -1\}$ is similar, and
on the finite interval $\{1 + \mathbf{i} t : -1 \leq t \leq
1\}$ the functions $E_i$ and $\Psi^{(U)}$
 are bounded independently of $U$. For the last claim, note that the
 calculation in \eqref{eq:tlessU} holds also for $0\leq t\leq1$, and
 the second to last equality there implies boundedness. In total we
 find  
\eq{eq: lem 1.14}{
\int_{\re(s) = 1} \Psi^{(U)}(s) E_i \left(g, \frac{s}{2}\right) R^s ds
= O\left(R \, U^{\frac12} \right).
}

\medskip

{\em Step 3. Justifying the contour shift.} We want to show that for any $R, U$, 
\eq{eq: what we want contour shift}{
\int_{\re(s) = \sigma >2} \Psi^{(U)}(s) E_i\left( g, \frac{s}{2}\right) R^s
ds = \sum_{\ell} c_\ell(U) R^{2s_\ell}  + \int_{\re(s) = 1} \Psi^{(U)}(s) E_i\left(
  g, \frac{s}{2}\right) R^s 
ds, 
} 
where $\Psi^{(U)}$ is the Mellin transform of $\psi^{(U)}$, the sum
ranges over the poles $(s_\ell)$ of 
$E_i(g, \cdot)$, and $c_\ell(U)$ are the corresponding residues. 
This follows from Proposition \ref{prop: for contour shift}, with $a
=1, b=\sigma>2$, $E(s) = E_i\left(g, \frac{s}{2}\right)$, and $\Psi(s) = R^s
\Psi^{(U)}(s).$ Note that by \equ{eq: choice of g}, for upper
bounds as needed for 
Proposition \ref{prop: for contour shift}, it makes no difference if
one works with the twisted Eisenstein series in \equ{eq: twisted
  eisenstein} or in \equ{eq:
  twisted eisenstein 2}. Hypotheses (i) and (ii) of Proposition \ref{prop: for
  contour shift} hold by Theorem \ref{thm:
  factsheet},  (iii) follows by
repeated 
integration by parts as in the proof of \equ{eq: lemma 1.13}, and so
we need to show (iv). The case $b = \sigma$ is trivial because 
$t \mapsto E(b + \mathbf{i}t)$ is bounded, and the case $\sigma=1$ was
proved in Step 2. 
Thus (iv) holds.   

\medskip

{\em Step 4. Combining bounds.} Using \equ{eq: what we want contour
  shift} with a main term $\sum_\ell c_\ell R^{2s_\ell}$, and matching the errors incurred
in \equ{eq: residue difference} and \equ{eq:
  lem 1.14} gives an error estimate 
$$
\max_\ell \frac{R^{2s_\ell}}{U} = \frac{R^2}{U} = R\, U^{\frac12}.
$$
This leads to a choice $U =
R^{\frac23}$ 
and the
combined error becomes $O\left(R^{\frac{4}{3}} \right)$. This
error is valid when using either one of $f^-_{R,U}$ and
$f^+_{R,U}$. Proposition \ref{prop: quadratic upper bound} implies that 
$$\mathbf{N}\left( g, \frac{R}{U}
\right) = O\left( \frac{R^2}{U^2}\right) = R^{\frac23 
}.$$ 
Thus 
appealing to \equ{eq: eisenstein mellin 3} and \equ{eq: above and below}
completes the proof. 
\end{proof}

\begin{remark}
We are grateful to Ze'ev Rudnick for explaining to us how to replace our earlier
result $O\left(R^{\frac{4}{3}+\vre} \right)$ with 
$O\left( R^{\frac{4}{3}}\right)$. 
Specifically, Rudnick suggested the use of dyadic
decomposition in Step 2.
\end{remark}

\begin{remark}
Any improvement in the bound {\bf (G1/2)} gives a corresponding improvement
in the error term. In fact, for the case $\Gamma = \SL_2(\Z)$, or its 
principal congruence subgroups, and $n=0$, one can replace the term
$T^2$ appearing in the right
hand side of {\bf (G1/2)} by $T$. Using this, and modifying \equ{eq: lemma
  1.13} to a bound $O\left(\frac{1}{|t|^{1+\vre}} \right)$ 
in Step 2, yields an error term
$O\left(R^{1+\vre} \right)$ for any $\vre>0$, in place of $O\left(
  R^{\frac{4}{3}} \right).$ \combarak{I don't think that the $\vre$
  can be shaved off in this case but maybe I missed something?} The 
recent papers \cite{Huang} and 
\cite{Nordentoft} contain sup norm bounds for
Eisenstein series for some arithmetic groups $\Gamma$ which are not
principal congruence subgroups. These bounds lead to improvements for {\bf
  (G1/2)}, and using them, one obtains a better estimate than
$R^{\frac{4}{3}}$ in \equ{eq: effective error Good} 
for the discrete orbits arising in these cases. 
\end{remark}

\section{More general shapes}\name{sec: new}
\subsection{Counting in smooth star shaped domains} 
We first state a more detailed version of the first part of Theorem \ref{thm: counting
  in shapes}, for counting in a smooth star shape. 
\combarak{Here I did not try to get rid of the $\vre$. Note that it is
  also used to ensure the convergence of the series in \equ{eq: expression for
    Theta 2}.}
\begin{theorem}\name{thm: smooth star shape}
Suppose $\Gamma$ is a non-uniform lattice in $G$ containing
$-\mathrm{Id}$, $Y = \Gamma v$ is a 
discrete orbit corresponding to the $i$-th cusp, and suppose that for
each $n$, $E_i(z,s)_{n}$ has trivial residual spectrum. Let $\rho: \R
\to \R_+$
be a smooth $2\pi$-periodic function, $S =
\{re^{\mathbf{i}\theta}: 0 \leq r \leq \rho(\theta)\}$, and let
$c_{Y,S} = \frac{\mathrm{vol}(S)}{\pi \, 
  \mathrm{covol}(\Gamma)}$. Then for every 
$\vre>0$,  
$$
\left| Y \cap R S \right| = c_{Y,S}R^2 +  O\left(R^{\frac{12}{7}+\vre}\right),
$$
where  the implicit constant depends on $\vre$ and $\rho$. 
\end{theorem}

\begin{proof}
The proof follows the same steps as in the proof of Theorem \ref{thm:
  more detailed good}. We define the same approximants $\psi^{\pm
  (U)}$ 
of the indicator function of the unit interval, so that
$x \mapsto \psi^{\pm(U)}\left(\frac{x}{\rho(\theta)}\right)$ are 
approximations of the indicator function of the interval $[0,
\rho(\theta)]$, in the sense of \equ{eq: approximation function
  properties}. Then we set  
$$
f^{\pm}_{R, U} (re^{\mathbf{i}\theta}) = \psi^{\pm(U)} \left(
  \frac{r}{\rho(\theta) \, R} \right),
$$
so that in analogy with \equ{eq: above and below}, we have 
\eq{eq: above and below 3}{
\Theta_{f^-_{R,U}} (g\Gamma) \leq \left|Y \cap RS \right| \leq
\Theta_{f^+_{R,U}}(g\Gamma) + 
N \left(g, R \, \frac{\max_\theta \rho(\theta) }{U} \right) \leq
\Theta_{f^+_{R,U}}(g\Gamma) + O\left(\frac{R^2}{U^2} \right) .
}
As before we continue with $f_{R,U}$ standing for one of the
$f^{\pm}_{R, U}$ and $\Psi^{(U)}$ standing for the Mellin transform of
one of the $\psi^{\pm(U)}$. Using \equ{eq: twisted eisenstein mellin 4} we have 
\eq{eq: expression for Theta}{
\Theta_{f_{R,U}}(g\Gamma) = \frac{1}{2\pi \mathbf{i}} \sum_{n\in \Z}
\int_{\re(s)=\sigma} \Psi^{(U)}(s) \hat{\rho}_{-n}(s) E_i\left( g,
  \frac{s}{2} \right)_{n} R^s \, ds \ \  \ (\text{where } \sigma >2).
}
Since we have assumed that $-\mathrm{Id} \in \Gamma$, the terms
corresponding to odd $n$ all vanish. 
For each $n \neq 0$, the functions $s \mapsto \Psi^{(U)}(s) \hat{\rho}_{-n}(s)
E_i\left(g,
  \frac{s}{2} \right)_{2n}$ are holomorphic on $\{s \in \C : \re(s)
\geq 1\}$ by our assumption that all of the $E_i(g,s)_{2n}$ have
trivial residual spectrum and by {\bf (P)} and {\bf (1)}. For $n=0$, the function $s
\mapsto \Psi^{(U)}(s) E_i\left(g, \frac{s}{2} \right) $ has a
simple pole at $s=2$, and by a computation as in the proof of \equ{eq: residue
  difference}, the residue $c_0(U)$ satisfies  
\eq{eq: residue difference 2}{
c_0 - c_0(U) = O \left( \frac{1}{U} \right).
}
Here $c_0 = \hat{\rho}_{0}(2)\, \mathrm{covol}(\Gamma)^{-1} $, and
since 
$\hat{\rho}_{0}(s) =\frac{1}{2\pi} 
\int_0^{2\pi} \rho(\theta)^2 \, d\theta$, computing the area of
$S$ in polar coordinates we obtain $c_0 =
\frac{\vol(S)}{\pi \, \mathrm{covol} (\Gamma)}$.

We now bound the integral of
$$s \mapsto h^{(U)}(s)E_i\left( g, \frac{s}{2} \right)_n, \ \ \text{
  where } h^{(U)}(s) = \Psi^{(U)}(s)  \hat{\rho}_{-n}(s),$$ 
along the critical line $\{\re(s)=1\}$, by a bound
depending on both $U$ and $n$. Thus from now on implicit constants may
depend on $\rho$ but not on $n$ and $U$. 
We will use parameters $k, \lambda, \vre$ which we will optimize further
below. 

For each $k \geq
0$ we have 
\eq{eq: lemma 1.13 4}{
\Psi^{(U)}(1 + \mathbf{i}t) \ll \frac{U^k}{|t|^{k+1}},
} 
and for each $\lambda \geq 0$ and $n \neq 0$, we have 
\eq{eq: upper bound fourier 3}{
|\hat{\rho}_{n}(s)| \ll \frac{|t|^{\lambda}}{|n|^{\lambda}}, \ \text{
  where } s = 1 + \mathbf{i}t.
}
Indeed, we get \equ{eq: lemma 1.13 4} for $|t| \geq U$ by
performing integration 
by parts $\lfloor k \rfloor +1 $ times (see \equ{eq: lemma 1.13}), 
and for $|t|\leq U$ by
applying integration by parts $\lfloor k \rfloor $ times. The proof of
\equ{eq: upper bound fourier 3} is similar. 
Using this we have 
\eq{eq: computation two parts}{
\begin{split} &\int_{\re(s)=1} h^{(U)}(s)
E_i\left(g,\tfrac{s}{2} \right)_n ds \\
& \ll
\int_1^\infty\frac{U^k}{t^{k+1}}\frac{t^\lambda}{|n|^\lambda}
E_i\left(g,\tfrac{1+\mathbf{i}t}{2} \right)_n dt \\ 
& \ll \frac{U^k}{|n|^{\lambda}} 
\left(\int_1^{\infty}
  \left(t^{1+\eps/2}\frac{t^\lambda}{t^{k+1}} \right)^2
  dt \right)^{1/2} \, \left(\int_1^\infty \frac{|E_i(g,\tfrac{1+\mathbf{i}t}{2})_n|^2}{
  t^{2+\eps}} dt \right)^{1/2}. 
\end{split}
}
To ensure finiteness of the first integral we will assume that 
\eq{eq: ensure finiteness}{
2k > 2\lambda + 1 + \vre.
}
For the second integral, we define $H(T) =\int_1^T \left| E_i\left(g,
    \frac{1 + \mathbf{i}t}{2}\right)_n\right|^2 dt$, so that 
{\bf (G1/2)} gives $H(T) \ll 
(T+n)^2$. Then integration by parts gives 
$$
\int_1^\infty \frac{|E\left(g, \frac{1+\mathbf{i}t}{2}\right)_n|^2}{t^{2+\vre}} \, dt \ll |n|^2.
$$
Using these estimates in \equ{eq: computation two parts} gives 
\eq{eq: lemma 1.14 3}{
\int_{\re(s)=1} \Psi^{(U)}(s) E_i\left(g, \frac{s}{2} \right)_n R^s 
\hat{\rho}_{-n}(s) \, ds \ll |n|^{1-\lambda} R U^k,
}
and the implicit constant depends on $\vre$. 

For each fixed $n$, the contour shift replacing the integral along
$\re(s)=\sigma$ with the integral along the line $\re(s)=1$ is
justified by Proposition \ref{prop: for contour shift} (note that in
condition (ii) of the Proposition, the implicit constants are allowed
to depend on $E$). We only pick up one residue, corresponding to $n=0$
and $s=2$. Thus 
collecting estimates we get:
\eq{eq: expression for Theta 2}{
\Theta_{f_{R,U}}(g\Gamma) =
c_0 R^2 + O \left(\frac{R^2}{U} \right) + \left( \sum_{n \in \Z}
  |n|^{1-\lambda} \right) \, O \left( RU^k \right) = c_0R^2 +
O\left(\frac{R^2}{U} \right) + O\left( RU^k\right),
}
where we set $\lambda =2+\vre$ to ensure convergence of the
sum. Setting $ k = \frac{5}{2} + 2\vre$ ensures \equ{eq: ensure
  finiteness}, and setting the two error terms equal to each other
gives $U = R^{\frac{1}{k+1}}$, which also ensures that that last term
in \equ{eq: above and below 3} is negligible. Thus \equ{eq: expression for Theta
  2} becomes 
$$
\Theta_{f_{R,U}}(g\Gamma) =
c_0 R^2 + O\left(R^{2-1/(\frac{7}{2}+2\vre)}\right),
$$
completing the proof. 
\end{proof}

\begin{remark}
1. As before, any improvement in the dependence on $n$, of the bound {\bf
  (G1/2)}, would lead to a corresponding improvement in the error
estimate. For $\Gamma = \SL_2(\Z)$ and principal congruence subgroups,
this improvement leads to an error estimate $O\left( R^{\frac{5}{3} + \vre}\right)$. 

2. We do not prove a version of Theorem \ref{thm: smooth star shape}
for lattices for which the twisted Eisenstein series has nontrivial
poles. If such poles $s_\ell$ existed, in performing the contour shift
argument, one would need to analyze the sum $\sum_{n \in \Z}
\mathrm{Res}|_{s=s_\ell} E_i\left(g, \frac{s}{2} \right)_{2n}$. As far as we
are aware, this series
is only known to be summable in the sense of distributions, and thus
analyzing it leads to technical issues we prefer not to enter into.

3. The assumption $-\mathrm{Id} \in \Gamma$ ensured that we only need
{\bf G(1/2)} for $n$ even, which is the context in which it was proved
in \cite{MS}. We are not aware of a proof of {\bf (G1/2)} in the
literature for $n$ odd. 
\end{remark}

For the proof of Theorem \ref{thm: counting in shapes} we will
need another construction which interestingly is also due to Selberg,
see \cite[Chap. 1, \S2]{Montgomery}. 
\begin{proposition}\name{prop: Selberg polynomials}
For each interval $J \subset \R$ and each $V \in \N$ there are
trigonometric polynomials $P^{\pm} = P^{\pm}_{J,V}$ such that 
\begin{enumerate}
\item
For all $x \in \R$, 
$$
P^{-}(x) \leq \chi_J(x) \leq P^+(x)
$$
(where $\chi_J$ is the indicator function of $J$);
\item
The degree of $P^{\pm}$ is at most
$V$;
\item
for each $0< |k| \leq V$, the $k$-th Fourier coefficient satisfies
$|\widehat{P^{\pm}}_k| \ll \frac{1}{k}$; and
\item
$\left|\left(\widehat{P^{\pm}}\right)_0 - \left(\widehat{\chi_j}\right)_0 \right|  \ll \frac{1}{V}.$
\end{enumerate}
\end{proposition}
We now state a more detailed version of the second part of Theorem
\ref{thm: counting 
  in shapes}, for counting in a sector. \combarak{Here Ze'ev's
  argument went through with no difficulties that I could
  detect. Please check.}

\begin{theorem}\name{thm: sector count}
Suppose $\Gamma$ is a non-uniform lattice in $G$, $Y = \Gamma v$ is a
discrete orbit corresponding to the $i$-th cusp, and suppose that for
each $n$, $E_i(z,s)_{n}$ has trivial residual spectrum. \combarak{Distinguish
  cases according as $-\mathrm{Id} \in \Gamma$}. Let $J
\subset \R$ be an interval of length $|J| \leq 2\pi$, let $S =
\{re^{\mathbf{i}\theta}: 0 \leq r \leq 1, \, \theta \in J\}$, and let $c_{Y,S} = \frac{|J|}{2\pi \,
  \mathrm{covol}(\Gamma)}$. Then 
\eq{eq: upper bound we will prove}{
\left| Y \cap R S \right| = c_{Y,S}R^2 +  O\left(R^{\frac{8}{5}
}\right),
}
where the implicit constant depends on 
$J$. \combarak{explicate dependence on $J$.}
\end{theorem}

\begin{proof}
We follow the same steps with the same notations, but now we introduce
an additional approximation parameter $V$, and let $\rho^{\pm 
  (V)}(\theta)$ be approximations of the indicator function $\chi_J$ of $J$,
namely they will satisfy 
\eq{eq: selberg polys requirement}{
\forall s \in \R/2\pi \Z, \ \ \rho^{-(V)}(\theta) \leq \chi_J(\theta)
\leq \rho^{+(V)}(\theta),
}
so that
$re^{\mathbf{i}\theta}  \mapsto \rho^{\pm(V)}(\theta)
\psi^{\pm(U)}\left(r \right)$ are 
approximations of the indicator function of $S$. Then we set  
$$
f^{\pm}_{R, U, V} (re^{\mathbf{i}\theta}) = \rho^{\pm(V)}(\theta) \, \psi^{\pm(U)} \left(
  \frac{r}{ R} \right),
$$
so that 
\eq{eq: above and below 4}{
\Theta_{f^-_{R,U,V}} (g\Gamma) \leq \left|Y \cap RS \right| 
\leq
\Theta_{f^+_{R,U,V}}(g\Gamma) + O\left(\frac{R^2}{U^2} \right) .
}
As before to lighten notation we omit the superscripts for upper and
lower bounds. Using \equ{eq: twisted eisenstein mellin} we have 
$$
\Theta_{f_{R,U, V}}(g\Gamma) = \frac{1}{2\pi \mathbf{i}} \sum_{n\in
  \Z} \hat{\rho}_{-n,V}
\int_{\re(s)=\sigma} \Psi^{(U)}(s) E_i\left( g,
  \frac{s}{2} \right)_{n} R^s \, ds,
$$
where $ \hat{\rho}_{n,V}$ is the $n$th Fourier coefficient of
$\rho^{(V)}$ and $\sigma>2$. 

For each $n$ we perform a contour shift to shift the integral to the line $\{\re(s)=1\},$
justifying it with Proposition \ref{prop: for contour shift}. By our
assumption on $\Gamma$, the only
residue $c_0(U,V)$ that we need to take into account  occurs for $n=0$ at $s
=2$. Setting $c_0 = \frac{|J|}{2\pi \mathrm{covol}(\Gamma)}$, the
residue satisfies
\eq{eq: residue difference 4}{
c_0 - c_0(U,V) = O\left( \frac{1}{U} + \left|\hat{\rho}_{0,V} - \frac{|J|}{2\pi} 
  \right| \right). 
}
Motivated by this, for the functions $\rho^{\pm(V)}$  we use the
polynomials $P^{\pm}_{J,V}$ of Proposition \ref{prop: Selberg
  polynomials} with $U=V$. With this choice \equ{eq: residue
  difference 4} becomes  
$$
c_0 - c_0(U,V) = O\left( \frac{1}{U} \right), 
$$
and we get a bound 
$$
\sum_{n \in \Z}  |n| |\hat{\rho}_{n,V}|  \ll \frac1U + \sum_{0<|n| \leq U}  |n|
\left| \frac1n \right| \ll U.
$$

We now repeat the arguments in Step 2 of the proof of Theorem \ref{thm: more detailed good}  to obtain
$$
\int_{\re(s)=1} \Psi^{(U)}(s)E_i\left( g, \frac{s}{2} \right)_n R^s \, ds
\ll |n| RU^{\frac{1}{2} 
}.
$$
Note the explicit
dependence on $n$ which arises by using {\bf (G1/2)} in \equ{eq: CS}. 

Collecting estimates we get
$$
\Theta_{f_{R,U}}(g\Gamma) =
c_0R^2 +
O\left(\frac{R^2}{U} \right) + O\left( R U^{\frac{3}{2} 
} \right),
$$
and equating the two error terms and plugging into \equ{eq: above and
  below 4} leads easily to \equ{eq: upper bound we will prove}. 
\end{proof}

\subsection{Counting in well-rounded sets }
In the present section we will prove Theorem \ref{thm:LWR}, using an
argument based on a general lattice point counting result.   
In order to state it, we recall the following:
\combarak{This definition was previously given in \S\ref{subsec: well
    rounded domain counting} but not
  used there at all. I moved it here. }

\begin{definition}(\cite{GN12})
\label{def: well--roundedness}Let $G$ be a connected Lie group with Haar
measure $m_{G}$. Assume $\left\{ G_{t}\right\} \subset G$ is a
family of bounded Borel sets of positive measure such that
$m_{G}\left(G_{t}\right)\ra\infty$ 
as $t\ra\infty$. Let $\cO_{\eta}\subset G$ be the image of a ball
of radius $\eta$ (with respect to the Cartan-Killing norm) in the Lie
algebra under the exponential map. Denote 
\[
G_{t}^{+}\left(\eta\right)=\cO_{\eta}G_{t}\cO_{\eta}=
\bigcup_{u,v\in\cO_{\eta}}u\,G_{t}\,v,\,\,\,\,\,\,   
G_{t}^{-}\left(\eta\right)=\bigcap_{u,v\in\cO_{\eta}}u\,G_{t}\,v.
\]
The family $\left\{ G_{t}\right\} $
is \emph{Lipschitz well-rounded} if there exist positive $c, \eta_{0}, t_{0}$
such that for every $0<\eta\leq\eta_{0}$ and $t\geq t_{0}$,
\[
m_{G}\left(G_{t}^{+}\left(\eta\right)\right)\leq
\left(1+c\eta\right)\,m_{G}\left(G_{t}^{-}\left(\eta\right)\right).  
\]
\end{definition}

\begin{theorem}\label{GN} \cite{GN12}  Let $G$ be any connected almost
  simple non-compact Lie group (e.g. $\SL_2(\R)$), and let $\{G_t\}$  
be a
  Lipschitz well-rounded family of subsets of $G$. 
Let $\Gamma$ be a lattice in $G$, and let
  $m_G$ be Haar measure on $G$, normalized so that
  $\mathrm{covol}(\Gamma)=1$. Define 
  the corresponding 
  averaging operators 
$$
\beta_t(f)(x) =\frac{1}{m_G(G_t)} \int_{G_t} f(g^{-1}x) dm_G(g), \ \ f \in
L_0^2(G/\Gamma). 
$$
Suppose the $\beta_t$ satisfy the following
  (operator-norm) bound:  
$$\norm{\beta_t}_{L^2_0(G/\Gamma)}\le Cm_G(G_t)^{-\kappa}\,. $$
Then the lattice point counting problem in $G_t$ has the effective solution 
$$\frac{\left|\Gamma \cap G_t \right| }{m_G(G_t)}=1+O\left(m_G(G_t)^{-\frac{\kappa}{\dim G+1}}\right).$$ 
\end{theorem}

The proof of Theorem \ref{thm:LWR} proceeds by reducing the problem of
counting points in the orbit $\Gamma v$ lying in bounded subsets of the plane, to
counting lattice points in suitable bounded domains in the group $G
=\SL_2(\R)$. The domains constructed in $ \SL_2(\R)$ bijectively  
cover the domains in the plane,
 under the orbit map $g \mapsto gv$, and will depend non-trivially on
 the orbit under consideration, and not just on $\Gamma$.  

In this section we write 
$$
a_t = \left(\begin{matrix} e^{t/2} & 0 \\ 0 & e^{-t/2}\end{matrix} \right).
$$
Let $v\in \R^2\sm \set{0}$ and use polar coordinates in the
plane to write $g=r_{\theta_v}a_{t_v}$ where
$v=g\mathbf{e}_1=r_{\theta_v}a_{t_v} \mathbf{e}_1=e^{t_v/2}r_{\theta_v}\mathbf{e}_1$. The
stability group of $v$ is $N^g=gNg^{-1}$, and for any $t$,
$x$ and $\theta$,
\begin{equation}\label{coordinates}
r_\theta (ga_tg^{-1})(gn_xg^{-1})(v)=r_\theta g a_t \mathbf{e}_1=
e^{t/2} r_\theta
v=e^{\frac12(t+t_v)}r_{\theta+\theta_v}\mathbf{e}_1\,.
\end{equation} 
 This gives a bijective parameterization of $\R^2 \sm \{0\}$ by $\R
 \times [0,2\pi)$, with each
 pair $(t,\theta) \in \R \times [0, 2\pi)$ determining a unique vector $r_\theta g a_t
 \mathbf{e}_1=e^{t/2}r_{\theta} v$ in $\bR^2\sm\set{0}$. Since
 $\bR^2\sm\set{0}=G/N^g$, we conclude that $G=KA^g N^g$, and
 this decomposition gives unique coordinates to each point in $ G$.
 Note however that this is {\it not} an Iwasawa decomposition, the
 latter being given by $G=K^gA^gN^g$.

 
Let us denote 
$$A_{t_1,t_2}=\set{a_t :  t_1\le t \le t_2}, \ 
N_{x_1,x_2}=\set{n_x :  x_1\le x \le x_2} \ \text{ and } 
K_{\theta_1,\theta_2}=\set{r_\theta :  \theta_1\le \theta \le
  \theta_2}. $$ 
%

Let $D\subset \bR^2$ be a compact set in the Euclidean plane given in
polar coordinates by  
$$D=\set{\rho (\cos\theta, \sin \theta)\,:\, \theta\in I\,\,,\,\,
 0 \le \rho \le \rho(\theta)}$$ 
where $I=[\theta_1,\theta_2]\subset [0,2\pi]$ is an interval of angles contained in
the unit circle, and $\rho(\theta)$ is a positive Lipschitz
continuous function on the interval $I$. 
The set $D$ can also be written in the form  
$$D=\set{r_\theta a_{t} \mathbf{e}_1\, : \, \theta\in
  I\,\,,\,\, t \le 2 \log \rho(\theta)} \cup \{0\}.$$ 
Let $b'>0$ be such that $\Gamma v $ contains no points of norm less than
$b'$. For any $T \ge 1$ consider as before the dilated set $T \cdot
D = \{T x : x \in D\}$, and  also the set
$$
D_T=\set{r_\theta a_{t} \mathbf{e}_1\,: \, \theta\in I\,\,,\,\, 2 \log b'
  \leq  t \le 2 \log(T\rho(\theta))}. 
$$
Then $D_T = T \cdot D \sm B(0, b'),$ and hence $|\Gamma v \cap
D_T| = | \Gamma v \cap T\cdot  D|$. 

We  will now define bounded domains $\widetilde{D}_T\subset G$ which
bijectively cover 
$D_T$. 
%
%
%
Fix a positive number $x_0=x_0(g)$ so that the set $N^g(x_0)=\set{gn_xg^{-1} : 0\le x
   < x_0}$ is a fundamental domain for the subgroup $\Gamma\cap
N^g\cong \bZ$ in the group $N^g\cong \bR$.  
For each $\theta\in I$ define 
 $$J(T,\theta)=\set{t\in \R\,: \, 2\log
   \frac{b'}{\norm{v}}\le t \le 2\log
   \frac{T\rho(\theta)}{\norm{v}}}=[t_1,t_2(T,\theta)]$$ 
and with respect to the decomposition $G=KA^gN^g$, define  
$$\widetilde{D}_{T}=\set{r_{\theta-\theta_v}ga_tg^{-1}\,: \, \theta\in
  I, t\in J(T,\theta) }\cdot  N^g(x_0)\,. 
$$

Then, as the reader may verify using \eqref{coordinates}, the orbit
map $G \to \R^2, g \mapsto g 
v $, restricted to $\Gamma \cap \widetilde{D}_T$, is a bijection with
its image 
$\Gamma v \cap D_T$, and as a consequence we obtain:
\begin{lemma}\label{lift} 
$| \Gamma v\cap T \cdot D|=|\Gamma \cap \widetilde{D}_{T}|$. 
\end{lemma} 

\ignore{
\begin{proof} 
We will show that for $\gamma\in \Gamma$, the conditions $\gamma v\in
D_T$ and  $\gamma\in  \widetilde{D}_{T}$ are equivalent. 

First let $\gamma\in  \Gamma \cap  \widetilde{D}_{T}$, so that
$\gamma=r_{\theta-\theta_v}ga_sg^{-1}n_x^g$ (with $\theta \in I$,
$s\in J(T,\theta)$, and $x\in [0,x_0)$), 
where this representation is unique since the decomposition
$G=KA^gN^g$ gives unique coordinates for each element of $G$. It
follows from the fact that $g=r_{\theta_v}a_{t_v}$ and $e^{\frac12
  t_v}=\norm{v}$ 
that $\gamma v \in D_T$. Conversely, let $\gamma v =r_\theta a_{s}
\mathbf{e}_1\in 
D_{T}$, with $\theta\in I$ and $2\log (\rho_1(\theta)) \le s \le 2
\log(T\rho(\theta))$.  Then $r_{\theta-\theta_v}ga_{s-2\log
  \norm{v}}g^{-1}v=r_\theta a_{s} \mathbf{e}_1=\gamma v$. Therefore
$\gamma$ is congruent to an element of the form 
$r_{\theta-\theta_v}ga_{s^\prime}g^{-1}$ modulo $N^g$, with $\theta\in
I$ and $s^\prime =s-2\log \norm{v}\in J(T,\theta)$, since $N^g$ is the  
stability group of $v$.  It follows that
$\gamma=r_{\theta-\theta_v}ga_sg^{-1}n_x^g$ for some $x\in \R$. But
$\gamma$ is congruent to a unique such element whose $N^g$ component
is in the fundamental domain of $\Gamma\cap N^g$ in $N^g$. It follows
that the two finite sets in question are in one-to-one correspondence.  
\end{proof}

}
To complete the proof of Theorem \ref{thm:LWR}, it remains to prove
that the family $\{\widetilde{D}_T\}$ is Lipschitz well-rounded. It will be convenient to
use the following two facts.  

\combarak{There was a fact here which was stated, about stability of
  the sets $\cO_\eta$ under conjugations. I believe this is not used
  anywhere so I commented it out. Please check. }
%
%
 
We will need the following result: 
 \begin{lemma}\label{well-rounded}(see \cite{GN12})
 If $\{G_t\}$ is a Lipschitz well-rounded family of subsets of $G$, then for each
 $g,h\in G$, so are the families  
$\{gG_t\}$, $\{G_t g\}$ and $\{gG_t h\}$. Furthermore the
corresponding constants $c, t_0, \eta_0$ are bounded above and away
from zero, as $g$ varies
on a compact set $Q$ in $G$.
\end{lemma} 

It therefore suffices to prove that the sets 
$$ \widetilde{D}_{T}g=\set{r_{\theta-\theta_v}ga_t n_x\,:\, \theta\in
  I, t \in J(T,\theta) , 0\le x <x_0}$$ 
are Lipschitz well-rounded. Recalling that $g=r_{\theta_v}a_{t_v}$ and
setting $T = e^{\tau}$, we have 
$$\widetilde{D}_T g = \cC_\tau=\set{r_{\theta}a_{t+t_v} n_x\,:\,
  \theta\in I, t\in J(e^\tau,\theta) , 0\le x < x_0}.$$ 

The Iwasawa coordinates $K\times A \times N\to G$ given by
$(k,a,n)\mapsto kan$  satisfy the following Lipschitz property,
established in   
\cite[Prop. 4.4]{HoreshNevo}. For every fixed $S_0\in \R$,
there exist $C_1=C_1(S_0)> 0$ and $\eta_1=\eta_1(S_0) > 0$, such
that for all $\theta$, all $x$ with $0\le x\le x_0$, and all
$t \ge S_0$, $0 < \eta < \eta_1$:  
\begin{equation}\label{perturb}
\cO_{\eta}r_\theta a_t n_x \cO_{\eta}\subset
K_{\theta-C_1\eta, \theta+C_1\eta}
A_{t-C_1\eta,t+C_1\eta}N_{x-C_1\eta,x+C_1\eta}\,. 
\end{equation}

Let  $S_0=2\log b'$, let $C_1=C_1(S_0)$, and let $C=C_1L$, where $L$ is Lipschitz
constant of the function $\rho$. Finally for  
$$\eta< \min\left \{ \frac{\eta_1}{4C},
  \frac{\theta_2-\theta_1}{4C}, \frac{x_0}{4C}\right\},$$ 
let 
$$
W^-(\tau,\eta) = \left\{k a n: k \in K_{\theta_1+ C\eta, \theta_2- C\eta}, \
a \in A_{t_1+ C\eta,t_2(e^\tau,\theta) - C\eta}, \ n \in N_{C\eta, x_0-C\eta} 
\right\}.$$
Applying (\ref{perturb}) to $g\in W^-(\tau,\eta)$ it follows readily that 
$W^-(\tau,\eta)\subset C_\tau^-(\eta). $
A straightforward verification, using the explicit form of Haar
measure in Iwasawa coordinates and the fact that $\rho$ is 
Lipschitz, shows that $m_G(W^-(\tau,\eta) ) \ge  (1-
c_1\eta) \cdot m_G(\cC_\tau)$, for a suitable $c_1 > 0$. In the other
direction, note that $C_t^+(\eta)=\cO_{\eta} \cC_t
\cO_{\eta}$ is contained in 
$$
W^+(\tau,\eta)=\left \{ k a n : k \in K_{\theta_1-C\eta,
    \theta_2+C\eta},\ a \in A_{t_1- C\eta,t_2(e^\tau,\theta)
    +C\eta} ,\ n \in N_{-C\eta , x_0+C\eta} \right \},$$ 
where for $\theta\in [\theta_1-C\eta, \theta_1] \cup
[\theta_2,\theta_2+C\eta]$ we define
$t_2(e^\tau,\theta)=\max (t_2(e^t,\theta_1), t_2(e^t,\theta_2))$.  
Again a similar direct verification shows that $m_G(W^+(\tau,\eta))\le
(1+c_2\eta) \cdot m_G(\cC_\tau)$. The Lipschitz well-roundedness of the
family $\cC_\tau$ follows, and this completes the proof of Theorem~\ref{thm:LWR}.   
\qed
%
%
%
%
\begin{remark}
\begin{enumerate}
%

\item Let us note that an error estimate established for the count in
  the dilates $R\cdot S$ of any  given figure (in the plane, say),
  immediately implies an error estimate for the count in {\em shells} of
  shrinking width, namely with the sets $R\cdot S\sm
  (R-R^{-\alpha})\cdot S$ for a suitable range of positive parameters
  $\alpha$. Similarly it is also possible to intersect shells with sectors
  of shrinking angle, namely with sets $\set{r(\cos
    \theta,\sin\theta) \,:\, \theta\in [\theta_0-R^{-\beta},
    \theta_0], r\in \R_+}$, for a suitable range of positive
  parameters $\beta$, and obtain an effective estimate. This follows
  from the fact that Theorem~\ref{GN} allows counting in a variable
  family of domains, provided that their Lipschitz well-roundedness
  parameters are controlled.  

\item A straightforward modification of the proof of
  Theorem~\ref{thm:LWR} applies to counting in discrete orbits of
  non-uniform lattices in $\SL_2(\C)$ acting linearly on $\C^2$, as
  well as discrete orbits of
  non-uniform lattices in $\SO^0(n,1)$ acting linearly on
  $\R^{n+1}$. This is based on Theorem~\ref{GN} and the Lipchitz
  property of the Iwasawa decomposition, established in
  \cite{HoreshNevo} for any non-exceptional group of real rank one.  
\end{enumerate}
\end{remark}

\begin{appendix}
\section{The growth estimate {\bf (G1)} for general $n$, and the
  function $\omega(t)$}
An important input to our argument is the estimate {\bf (G1)} which is
used to bound the average growth of the Eisenstein series along vertical lines
$\re(s) = \sigma$, for $\sigma \in [1/2,1]$. This is a crucial input
to our method, see condition (ii) of
Proposition \ref{prop: for contour shift}. This growth estimate was
proved by Selberg for $n=0$ but as far as we are aware, does not
appear in the literature for the twisted Eisenstein series for general
$n$. In this appendix we close this gap in the literature, and also
provide estimates of the dependence of the implicit constant on
$n$. Many of our arguments are based on ideas in \cite{Cohen Sarnak, Sarnak
  horocycles}. 

We first
introduce standard notation. 
Let
$\Phi(s) = \left(\varphi_{ij}(s) \right)_{ij}$ be the constant term
matrix as in \equ{eq: constant term}, let $(s_\ell)$ denote all the
poles of the functions $\varphi_{ij}$ in the interval $[1/2,1]$, 
let $q = q_\Gamma>0$ be a real number specified by Selberg (see
\cite[p. 655]{Selberg gottingen}) and set  
\eq{eq: def omega}{
\Psi_0(s) = \det \Phi(s),
\ \ \Psi^*(s) = q^{2s-1} \prod_\ell \frac{s-s_\ell}{s-1+s_\ell} \Psi_0(s), \ \ 
\omega(t) = 1-\frac{\Psi^{*'}}{\Psi^*} \left(\frac12 + \mathbf{i} t \right). 
}
The function $\omega: \R \to \R$ thus defined is the function
appearing in Theorem~\ref{thm: factsheet}. It satisfies $\omega(t) >1$
for all $t$ (see \cite[p. 656]{Selberg gottingen}).

\begin{theorem}\name{prop: in appendix}
For any non-uniform lattice $\Gamma$ in $G$,
the twisted Eisenstein series $E_i( \cdot)_n$ corresponding to the $i$-th cusp as
in \equ{eq: twisted eisenstein 2} satisfies 
\eq{eq: want appendix}{
	\re{s}\geq\frac12, |t| \geq |n|+1 \ \ \implies \ \ 
		E_{i}(z,s)_{2n} \ll |t| \, 
\sqrt{\omega(t)} \ \ \ 
 (\text{where } t = \im(s)),
}
where implicit constants depend on $\Gamma$ and $z$ (but not on $n$). 
\end{theorem} 

\begin{proof}
We will divide the proof into a series of steps. Throughout the
proof we will write 
$$
\sigma = \re(s), \, t = \im(s), \ \ \text{ and } \ 
\vre = \sigma - \frac{1}{2},
$$ 
and will assume $\sigma \in \left[ \frac12, \frac32 \right],$ which
entails no loss of generality as $E_i(z, \cdot)_{2n}$ is uniformly bounded
on $\left\{s: \re(s) \geq \frac32 \right\}$.  Implicit constants in
the $\ll$ and $O$ notation depend on $\Gamma$ and $z$ (and not on $n$
or $\sigma$). Since some of our arguments will depend on the
dependence of the Eisenstein series on the variable $z$, from now on
we will write $z_0$ instead of $z$ and consider it as a fixed element
of $\Hyp.$ 

\medskip

{\em Step 1. Bounding $\Phi$ and the $\mathbf{\Gamma}$-factor.}
Let $\Phi(s) = (\varphi_{ij}(s))_{i,j}$ be the constant term matrix as
in \equ{eq: constant term}. By \cite[p. 655]{Selberg gottingen}, $\Phi$
is uniformly bounded as long as  $\frac{1}{2} \leq \sigma \leq
\frac{3}{2}$ and $|t| \geq 1$. Following \cite[Chapter 6]{Kubota},
define
$$
\Phi_{2n}(s) = (\varphi_{ij}(s)_{2n})_{i,j}  \ \ \text{ by } 	\ \
\Phi_{2n}(s)=(-1)^n \, B_N(s)\Phi(s), 
$$
where 
	$$
		B_n(s)=\frac{\mathbf{\Gamma}(s)^2}{\mathbf{\Gamma}(s-n) \mathbf{\Gamma}(s+n)}.    
	$$
		Then 
\eq{eq: step 1.1}{|t |\geq 1, \ \sigma \leq 1
 \ \implies 
		|B_n(s)|\leq 1,
}
with equality for $\sigma = \frac12$, and 
\eq{eq: step 2}{|t| \geq |n| \ \ \implies \ \ 
							1-|B_n(s)|^2
                                                        \ll  \vre.
			}

\medskip

To see this, since $B_n(s) = B_{-n}(s)$ we can assume that $n>0$. Using the recurrence formula
$\mathbf{\Gamma}(z+1)=z\mathbf{\Gamma}(z)$ one obtains
				\eq{eq: formula B}{
				B_n(s) 
			=	\frac{\mathbf{\Gamma}(s-n)^2\left((s-1)\dots(s-n) \right)^2}{(s-n)\dots(s+n-1)\mathbf{\Gamma}(s-n)^2}
=\prod_{k=1}^n
                                \frac{s-k}{s+k-1}.
				}
Since $\sigma \geq \frac{1}{2}$, $|s-k| \leq |s+k-1|$ with
equality when $\sigma = \frac{1}{2}$. This implies \equ{eq: step
  1.1}. 

	For \equ{eq: step 2}, 
       set $z_{s,k}=s+k-1 = k-\frac12+\eps + \mathbf{i}t$. Plugging
        into \equ{eq: formula B} gives 
				\[
\begin{split}
					|B_n(s)|^2 &=\prod_{k=1}^n\left|
                                          \frac{-\overline{z_{s,k}}+2\eps}{z_{s,k}} \right|^2
                                        = \prod_{k=1}^n \left(
                                          \frac{|z_{s,k}|^2 -4\vre \re(z_{s,k})+4\vre^2}
                                          {|z_{s,k}|^2}\right) \\
 &                                       =\prod_{k=1}^n \left(1- 4 \vre
                                          \,
                                          \frac{\re(z_{s,k})-\vre}{|z_{s,k}|^2}
                                        \right) = \prod_{k=1}^\mathbf{N}\left(1-2\eps\frac{2k-1}{|z_{s,k}|^2}\right)
.
			\end{split}	\]
Write
$$f(k) =
\frac{2k-1}{|z_{s,k}|^2}, \ \text{ so that  } \ F(n) = 
\sum_{k=1}^n
\log (1-2\vre f(k)) = \log \left( |B_n(s)|^2\right).$$
Let $k \leq n$ and $|t| \geq n$, then $f(k) \leq 
\frac{k}{n^2}
$. Taking a second
 order Taylor approximation for $x \mapsto \log(1-x)$ we have 
$$
-\log(1-2\vre f(k)) = 2\vre f(k) + O\left(\vre^2 f(k)^2 \right), 
$$
and hence 
$$-F(n) \ll \sum_{k=1}^n \vre f(k) + \vre^2
  f(k)^2 \ll \vre \sum_{k=1}^n \frac{k}{n^2} \ll \vre. $$
Now by second order Taylor approximation for $x \mapsto 1-e^x$ we get 
$$
1- |B_n(s)|^2 = 1- e^{F(n)} \ll - F(n) \ll \vre, 
$$
and we have shown \equ{eq: step 2}. 

\medskip

{\em Step 2. Regularized Eisenstein Series.}
We
choose a parameter $Y$ depending on $z_0$ by
\eq{eq: choice of Y}{
Y = 1+ \max_j \im(\mathfrak{s}_j^{-1}z_0), 
}
and define a regularized Eisenstein series  
		\eq{eq: regularized}{
	E_i^{Y}(z,s)_{2n}= \left\{ \begin{matrix}  E_{i}(z,s)
                  _{2n}-\delta_{ij}y_j^s-\varphi_{ij}(s)_{2n} \, y_j^{1-s}
                  & \ \ \text{ if } y_j=\im(\mathfrak{s}_j^{-1} z)\geq Y \text{
                    for some } j\\ 
E_{i}(z,s) _{2n} & \ \ \text{ otherwise}
\end{matrix} \right. 
		}
\combarak{Rene had $\sigma_j$ and I think $\sigma_j^{-1}$ is correct
  with our conventions, please check.}
(note that 
the condition $y_j \ge Y$ can occur for at most one index $j$). 
Let
$\cE^Y(z,s)_{2n}=\left(E^Y_1(z,s)_{2n},\dots,E^Y_k(z,s)_{2n}\right)^{\on{tr}} $. 
Then  
(see
e.g. \cite[p. 727]{Sarnak 
  horocycles}) for $\sigma > \frac12, \, t \neq 0$ 
we have the following inner product formula: 
\[ \begin{split} 
	& \int_{X_\Gamma}
        \cE^Y(z,s)_{2n}\overline{\cE^Y(z,s)_{2n}}^{\on{tr}} \, d\mu_{\Gamma}(z)\\
= &\frac{1}{2\vre}
\left(
Y^{2\vre} \on{Id}_{k\times k}
-\Phi_{2n}(s
)
\overline{\Phi_{2n}(s
)}^{\on{tr}} Y^{-2\vre}
\right )
+\frac{\overline{\Phi_{2n}(s
)}^{\on{tr}} Y^{2\mathbf{i}t}-\Phi_{2n}(s
) Y^{-2\mathbf{i}t}}{2\mathbf{i}t}.
\end{split}\]

\medskip

{\em Step 3. Trace of Maass-Selberg Relations.} 
 Let $\left\|E^Y_j(\cdot 
,s)_{2n} \right\|^2_2 = 
\int_{X_\Gamma} \left|E_j^Y(z,s)_{2n} \right|^2 \, d\mu_{\Gamma}(z)$,
$\|\Phi_{2n}(s)\|^2= \sum_{ij}  
|\varphi_{ij}(s)_{2n}|^2$,
	and	set $\Psi_{2n}(s)=\det{\Phi_{2n}(s)}$. Then for
                $\sigma > \frac12, \, t \neq 0$ we have 
		\begin{equation}
		\label{eq:trace} \begin{split}
		& \sum_{j=1}^{k}
\left\|E^Y_j(\cdot
,s)_{2n} \right\|^2_2 \\
             =   & \frac{1}{2\vre}\left(k
               Y^{2\vre}- \|\Phi_{2n}(s)\|^2 \, Y^{-2\vre} 
\right) 
		+ 
                \frac{1}{2\mathbf{i}t} \left(Y^{2\mathbf{i}t}\sum_{i=1}^k\overline{\varphi_{ii}}(s)_{2n}-Y^{-2\mathbf{i}t}\sum_{i=1}^k\varphi_{ii}(s)_{2n}\right). 
\end{split}
\end{equation}
	Indeed, this is a matrix computation that involves taking the
        trace of the inner product formula. See
        \cite[p. 140]{Iwaniec} for the computation in case
        $n=0$. \combarak{I did not check this. }

\medskip

{\em Step 4. Bounds on traces and norms.}
For $n=0$ and $|t| \geq 1$  we have 
		\eq{eq: step 4.1}{
		k-\|\Phi(s)\|^2 \ll \vre \,
                  \omega(t),
		}
and for $n \in \N, \, \sigma \in \left[\frac12, \frac32 \right]$ and
$|t| \geq n$ we also have 
		\eq{eq: step 4.2}{
			k-\|\Phi_{2n}(s)\|^2 \ll  \vre 
                          \, \omega(t).
		}
Indeed, \equ{eq: step 4.1} is proved in \cite[p. 657]{Selberg
  gottingen}, and since 
$$\|\Phi_{2n}(s)\|^2 = |B_n(s)|^2 \| \Phi(s) \|^2,$$ 
\equ{eq: step 4.2} follows from \equ{eq: step 4.1}, the boundedness of
$\Phi$, and \equ{eq:
  step 2}.

As to $L^2$ bounds, for $n\in \Z$ and $\sigma \in \left[\frac12, \frac32\right]$  we have
			\eq{eq: bounds norms1}{
			\left\|E_{j}^{Y}(\cdot, s)_{2n}
                        \right\|^2_2\ll \omega(t),
			}
where the implicit constant depends also on $Y$, and hence on $z_0$. 
\combarak{There used to be a reference here to Sarnak for the case $n=0$, but
  our proof works in general}. For this we use 
$
			Y^{\pm 2\vre}=1+O_Y(\vre).
$
Using \equ{eq: step 1.1} we have that $\Phi_{2n}(s)$ is uniformly bounded for
$\sigma \in \left[\frac12, \frac32\right]$, and hence the right hand 
summand in \eqref{eq:trace} is bounded. For the left hand summand, we
have by \equ{eq: step 4.2}
	\[ 
			\frac{1}{2\vre}\left(kY^{2\vre}-Y^{-2\vre}\|\Phi_{2n}(s)\|^2\right)
                        =\frac{1}{2\vre}\left(k -\|\Phi_{2n}(s)\|^2
                          +O_Y(\vre)\right) 
			=O_Y(\omega{(t)} + 1).
\]
Combining bounds and recalling $\omega(t)>1$ gives \equ{eq: bounds
  norms1} for $\sigma > \frac12$. Since the implicit constant in
\equ{eq: bounds norms1} is independent of
$\sigma$ we can take a limit and get the same bound for $\sigma = \frac12$. 

\medskip

{\em Step 5. Convolution and point-pair invariant.}
For $\gamma = \left( \begin{matrix} a & b \\ c & d \end{matrix}
\right) \in \Gamma$, $n \in \Z_{\geq 0}$ and $z \in \Hyp$, let 
$$
\eps_\gamma(z)^{2n} = \frac{(cz+d)^{2n}}{|cz+d|^{2n}}.
$$
Say a function $f$ on $\Hyp$ is {\em of weight $2n$} if it
transforms like
$f(\gamma z)=\eps_\gamma(z)^{2n}f(z)$, and denote the automorphic
functions, whose restriction to a Dirichlet fundamental domain for
$\Gamma$ on $\Hyp$ is in $L^2 (X_\Gamma, \mu_\Gamma)$, by $L^2(\Gamma,
2n)$. 	For $\delta \in (0,1)$, let $\chi_\delta$ be the indicator
function of $[0, \delta]$ \combarak{I took this from \cite{MS} but
  maybe we need $\chi_\delta$ smooth?}, and for $z, w \in \Hyp$ define
			\[ \begin{split}
			u(z,w)= & \frac{|z-w|^2}{4\, \im{z}\, \im{w}}
			\\
			H(z,w)=& (-1)^n\frac{(w-\overline{z})^{2n}}{|w-\overline{z}|^{2n}}
			\\
			k(z,w)= k_\delta(z,w) = & H(z,w)\chi_\delta(u(z,w))
			\\
			K(z,w)= K_\delta(z,w) = & \sum_{\gamma\in\Gamma}k_\delta(z,\gamma w)\eps_\gamma(z)^{2n}.
			\end{split} \]
Functions such as $k$ are called {\em point pair invariants of weight
  $2n$.} They satisfy (see \cite[Vol. 1, Prop. 2.11]{Hejhal}) the following
transformation rules:
			\[
			H(\gamma z,\gamma w)=\eps_\gamma(z)^{2n}H(z,w)\eps_\gamma(w)^{-2n}
			\]
			\[
			k(\gamma z,\gamma w)=\eps_\gamma(z)^{2n}k(z,w)\eps_\gamma(w)^{-2n}.
			\]
	The  operator $L_k$ defined by 
			\eq{eq: operator}{
			L_kf(z)=\int_{\bH}k(z,w)f(w)dw=\int_{\bH/\Gamma}K(z,w)f(w)dw,
			}
is a bounded self-adjoint operator on  
                $L^2(\Gamma, 2n)$, 
			see \cite[Vol. 1, Prop 2.13]{Hejhal}. Let
                        $\Delta$ be as in
                        \equ{eq: Laplacian}, and let 
                        $$\Delta_n u (z)  = \Delta u (z) + \mathbf{i}n  y
                        \frac{\partial u}{\partial x}(z) \ \
                        (\text{where } z = x +
                        \mathbf{i}y) $$
be the {\em weighted Laplacian.} \combarak{Note our Laplacian is the
  negative of the one in Hejhal.} Then the Eisenstein series $z \mapsto
E_i(z,s)_{2n}$ is a $\Delta_{2n}$-eigenfunction and therefore (see
\cite[Vol. 1, Prop. 2.14]{Hejhal}) is an eigenfunction for $L_k$, that
is there is $h_{i,n, \delta}(s)$ such that for all $z \in \Hyp$, $L_k E_i(z, s)_{2n} =
h_{i,n, \delta}(s) E_i(z, s)_{2n} $. 

\medskip

{\em Step 6. Bounding the eigenvalue.} 
The eigenvalue $h_{i,n, \delta}(s)$ satisfies a bound 
\eq{eq: MS}{\sigma \in \left[ \frac12, \frac32 \right], \ |t| \geq |n|+1, \
\delta = \frac{1}{100|t|^{2}} \ \implies 
		|h_{i,n, \delta}(s)|\gg \frac{1}{|t|^2}.	
}
Indeed, the bound \equ{eq: MS} is proved in \cite[Lemma 2.1]{MS} for
$\sigma = \frac12$, and the proof goes through for general $\sigma \in
\left[ \frac12, \frac32 \right]$. \combarak{Are more
  details needed? Again, I did not check this. }

\medskip

{\em Step 7. Pointwise bounds.}
We now note that our choice \equ{eq: choice of Y} implies that 
\eq{eq: can substitute}{
L_k E_i^Y(z_0, s)_{2n} = L_k E_i(z_0 , s)_{2n}.
}
Indeed, considering \equ{eq: regularized} and the definitions of
$\delta$ and $u$ we see that $E_i^Y(w, s)_{2n}$ and $E_i(w, s)_{2n}$
coincide for all $w$ in the neighborhood 
of $z_0$ consisting of the points for which the integrand in \equ{eq:
  operator} is nonzero.  

To conclude the proof of \equ{eq: want appendix}, we apply 
			Cauchy-Schwartz to find
			\[ \begin{split}
		\left|	E_{i}(z_0,s)_{2n} \right|  = &
                \frac{1}{|h_{i,n, \delta}(s)|} \left| L_k
                        E_i(z_0, s)_{2n} \right| \stackrel{\equ{eq: can
                            substitute}}{=} \frac{1}{|h_{i,n,
                            \delta}(s)|} \left| L_k
                        E^Y_i(z_0, s)_{2n}  \right | 
\\
\stackrel{\equ{eq: MS}}{\ll} & |t|^2 \, \left |
  \int_{\bH/\Gamma}K_\delta(z_0,w)E^Y_{i}(w,s)_{2n} \, dw \right| 
			 \\ \stackrel{\equ{eq: bounds norms1}}{\ll} &
			|t|^2 \, 
			\sqrt{\omega(t)} \,
 \sqrt{\int_{\bH/\Gamma} \left|K_\delta(z_0,w)\right|^2 \, dw} \ll 	|t|^2 \, 
			\sqrt{\omega(t)} \, \sqrt{\delta} \ll |t| \, \sqrt{\omega(t)}. 
					\end{split} \]
\combarak{This takes care of $n \gg |t|$, explain what to do for small
  $n$. }
\end{proof}

\end{appendix}
\bibliographystyle{alpha}

\begin{thebibliography}{}

\bibitem[AG13]{AG} 
A. Avila and S. Gou\"ezel, {\em 
Small eigenvalues of the Laplacian for algebraic measures in moduli
space, and mixing properties of the Teichm\"uller flow},
Ann. Math. {\bf 178} (2013) 385--442. 


\bibitem[AGY06]{AGY}
A. ~Avila, S. ~Gou{\"e}zel, and J-C. ~Yoccoz, {\em Exponential mixing
  for the {T}eichm\"uller flow,} Publ. Math. Inst. Hautes \'Etudes
Sci., (104):143--211, 2006. 

\bibitem[Be11]{bergeron} N. Bergeron, {\bf The spectrum of
    hyperbolic surfaces}, Universitext, Springer (2011). 

\bibitem[BM10]{BM} 
I. Bouw and M. M\"oller, {\em 
Teichm\"uller  curves,  triangle  groups, and Lyapunov exponents}, 
Ann. of Math. (2) {\bf 172}  (2010), no. 1, 139--185.

\bibitem[CS80]{Cohen Sarnak} P. Cohen and P. Sarnak, {\em Notes on the
  Selberg
  trace formula}, unpublished notes (1980), available at {\tt
  https://publications.ias.edu/sarnak/paper/496}  
  
\bibitem[Dav80]{Davenport} H. Davenport, {\bf Multiplicative number
    theory,} second edition, Springer (1980). 
 

\bibitem[DRS93]{DRS}
W.~Duke, Z.~Rudnick, and P.~Sarnak.
\newblock Density of integer points on affine homogeneous varieties.
\newblock {\em Duke Math. J.}, 71(1):143--179, 1993.

\ignore{
\bibitem[Strom04]{strombergssen} {On the uniform equidistribution of
    long closed horocycle}

\bibitem[Ed18]{Edwards} S. C. Edwards, {\em Renormalization of
    integrals of Eisenstein series and analytic continuation of
    representations}, preprint (2018) {\tt https://arxiv.org/abs/1809.07605
} 
}

\bibitem[EM01]{EskinMasur} A. Eskin and H. Masur, {\em 
Asymptotic formulas on flat surfaces, } Ergodic Theory Dynam. Systems
{\bf 21(2)} 443--478 (2001).

\bibitem[EMc93]{EM} A. Eskin and C. T. McMullen, {\em 
Mixing, counting and equidistribution in Lie groups}, Duke
Math. J. {\bf 71} (1993), no. 1, 181--209.

\bibitem[Go83]{Good} A. Good, {\bf Local analysis of Selberg's trace
    formula}, Springer Lecture Notes {\bf 1040} (1983).  



\bibitem[GN12]{GN12}
A.~Gorodnik and A.~Nevo, {\bf Counting lattice points.}
Journal f{\"u}r die reine und angewandte Mathematik,
  663:127-176, 2012.

\bibitem[GJ00]{GJ} E. Gutkin and C. Judge, {\em Affine mappings of
    translation surfaces: geometry and arithmetic}, Duke
  Math. J. {\bf 103} (2000), no. 2, 191--213. 


\bibitem[Hej83]{Hejhal} D. Hejhal, {\bf The Selberg trace formula, Volume
    1, 2}. Springer Lecture Notes, 1983.


\bibitem[Hoo13]{Hooper} W. P. Hooper, {\em 
Grid graphs and lattice surfaces,} IMRN {\bf 12} (2013) 2657--2698.



\bibitem[HoN17]{HoreshNevo} T. Horesh and A. Nevo, {\em Horospherical
    coordinates of lattice points in hyperbolic spaces: effective
    counting and equidistribution}, preprint (2016) 
{\tt https://arxiv.org/abs/1612.08215}  

\bibitem[HX16]{Huang} B. Huang and Z. Xu, {\em Sup norm bounds for
    Eisenstein series}, preprint (2016) {\tt https://export.arxiv.org/pdf/1508.02799v3} 

\bibitem[HN96]{HN} M. N. Huxley and W. G. Nowak, {\em Primitive
    lattice points in convex planar domains}, Acta Arith. {\bf LXXVI}
  (1996) 271--283.

\bibitem[IKKN06]{Ivic} A. Ivi\'c, E. Kr\"atzel, M. K\"uhleitner and
  W. G. Nowak, \textit{Lattice points in large regions and related
    arithmetic functions: Recent developments in a very classic
    topic}, Elementare und analytische Zahlentheorie,
  Schr. Wiss. Ges. Johann Wolfgang Goethe Univ., Frankfurt am Main,
  \textbf{20} (2006), 89--128.  


 
\bibitem[Iwa95]{Iwaniec} H.~Iwaniec, {\bf Introduction to the spectral
    theory of automorphic forms.} Biblioteca de la Revista Matematica
  Iberoamericana, Madrid, 1995. 
\ignore{
\bibitem[ET]{ET} P. Erd\H{o}s, P. Tur\'an, {\em On a problem in the
    theory of uniform distribution.} (1948)

}
\bibitem[Kub73]{Kubota} T. Kubota, {\bf Elementary theory of Eisenstein
    series}. Wiley $\&$ Sons, New York (1973).

\bibitem[MS03]{MS} J Marklof and A. Str\"ombergsson, {\em
Equidistribution of Kronecker sequences along closed horocycles, 
}  Geom. and Func. Ana. {\bf 13} (2003) 1239--1280.  

\bibitem[MT02]{MT}
H. Masur and S. Tabachnikov, {\em 
Rational billiards and  flat structures}, in {\bf Handbook of
dynamical systems,} Vol. 1A, 1015--1089. North-Holland, Amsterdam, 2002

\bibitem[M94]{Montgomery} H. L. Montgomery, {\bf Ten lectures on the
    interface between analytic number theory and harmonic analysis},
  AMS CBMS Lecture series {\bf 84} (1994).   

\bibitem[Ne17]{Nevo17} A. Nevo {\em Equidistribution in
    measure-preserving actions of semisimple groups: case of
    $\SL_2(\R)$} preprint (2017), {\tt Math. ArXiv:1708.03886}

\bibitem[NRW17]{NRW} A. Nevo, R. R\"uhr and B. Weiss, {\em  Effective
    counting on translation surfaces, } preprint 
  (2017). 

\bibitem[Nor19]{Nordentoft} A. C. Nordentoft, {\em Hybrid subconvexity
    for class group $L$-functions and uniform sup norm bounds of
    Eisenstein series}, preprint (2019) {\tt
    https://arxiv.org/pdf/1903.03932.pdf
}

\bibitem[P11]{Poincare} H. Poincar\'e, {\em Fonctions modulaires et
    fonctions fuchsiennes,} Ann. Fac. Sc. Toulouse (1911), 125--149.  

\bibitem[Sa80]{Sarnak thesis} P. Sarnak, {\em Prime geodesic
    theorems}, Stanford University Ph.D. dissertation (1980).

\bibitem[Sa81]{Sarnak horocycles} P. Sarnak, {\em Asymptotic behavior 
    of periodic orbits of the horocycle flow and Eisenstein series.}
  Comm. Math. Phys. (1981), 719--739. 


\bibitem[Sa03]{Sarnak survey} P. Sarnak, {\em Spectra of hyperbolic
    surfaces}, Bull. Amer. Math. Soc. {\bf 40} Number 4 (2003) 441--478.




\ignore{
\bibitem[Sel]{Sel} A. Selberg, {\em Harmonic analysis and discontinuous groups in weakly symmetric Riemannian spaces with applications to Dirichlet series}/ J. Indian Math. Soc {\bf 20} (1956), 47--87.
}

\bibitem[Se56]{Selberg Indian} A. Selberg, {\em 
 Harmonic analysis and discontinuous groups in weakly symmetric
 Riemannian spaces with applications to Dirichlet series}, J. Indian
Math. Soc {\bf 20} (1956) 47--87.

\bibitem[Se89]{Selberg gottingen} A. Selberg, {\em Harmonic analysis: G\"ottingen lecture
    notes} in {\bf Collected works}, Springer (1989). 

\bibitem[Ter85]{Terras} A. Terras, {\bf Harmonic analysis on symmetric
    spaces and applications}, Vol. 1 (1985) Springer Verlag. 

\bibitem[Vee89]{Veech Eisenstein} 
W. A. ~Veech,
\newblock {\em Teichm\"uller curves in moduli space, Eisenstein series
  and an application to triangular billiards}, 
\newblock Invent. Math. {\bf 97}  (1989), 553--583.

\bibitem[Vee92]{Veech regular}
W. A. ~Veech, {\em The billiards in a regular polygon},
Geom. Funct. Anal. {\bf 2} (1992) 341--379.

\bibitem[Vee98]{Veech Siegel}
W. A. ~Veech,
\newblock {\em Siegel measures,}
\newblock Ann. of Math.(2), {\bf 148(3)}: 895--944, 1998.

\bibitem[Vor96]{Vorobets}
Y. B. Vorobets, {\em Planar structures and billiards in rational
  polygons: the Veech alternative,} (Russian) Uspekhi Mat. Nauk. {\bf
  51} (1996) 3--42; translation in Russian Math. Surveys {\bf
  51} (1996) 779--817. 

\bibitem[W46]{Widder} D. V. Widder, {\bf The Laplace transform,}
  Princeton Univ. Press (1946). 


\bibitem[Wr13]{Wright Schwarz groups} A. Wright, {\em  Schwarz
    triangle mappings and Teichm\"uller curves: the
    Veech-Ward-Bouw-M\"oller curves}, Geom. Func. Anal. {\bf 23}
  (2013) 2, 776-–809. 

\bibitem[Z82]{Zograf} P. Zograf, {\em Fuchsian groups and small
    eigenvalues of the Laplace operator}, translated from 
Zap. Nauch. Sem. Leningrad. Otdel. Mat. Ins. Steklova {\bf 122} 24--29 (1982).

\ignore{
\bibitem[Wa98]{Ward}
C. C. Ward, {\em 
Calculation of Fuchsian groups associated to billiards in a rational triangle
},  Erg. Th. Dyn. Sys. {\bf 18} (1998) 1019--1042.
}

 
\bibitem[Z06]{Zorich survey}
A. Zorich, {\em Flat surfaces,} In
{\bf Frontiers in number theory, physics, and geometry.} Vol.
I, pages 437--583. Springer, Berlin, 2006. 

\end{thebibliography}

\end{document}